\documentclass[preprint]{imsart}
\usepackage[official]{eurosym}
\usepackage{amssymb,amsmath}
\usepackage{epic,eepic}
\usepackage{array}
\usepackage{amsthm}
\usepackage{graphicx}
\usepackage[top=3cm, bottom=5cm, left=3cm, right=3cm]{geometry}
\usepackage{imsart}
\usepackage{fullpage,amsmath}
\usepackage{amssymb}
\usepackage{bbding}
\usepackage{graphicx}

\newtheorem{theorem}{Theorem}
\newtheorem{lemma}[theorem]{Lemma}
\newtheorem{corollary}[theorem]{Corollary}
\newtheorem{definition}[theorem]{Definition}
\newtheorem{remark}[theorem]{Remark}
\newtheorem{example}[theorem]{Example}
\newtheorem{listprop}[theorem]{List of Properties}
\numberwithin{theorem}{section} \numberwithin{equation}{section}

\begin{document}

\title{Subdivision schemes of sets and the approximation of set-valued functions in the symmetric difference metric\thanksref{T1}}
\runtitle{Subdivision schemes of sets and the approximation of
set-valued functions in the symmetric difference metric}
\author{\fnms{Shay} \snm{Kels}\thanksref{t1}\ead[label=e1]{kelsisha@tau.ac.il}}
\and
\author{\fnms{Nira} \snm{Dyn}\ead[label=e2]{niradyn@tau.ac.il}}
\address{
School of Mathematical Sciences, Tel-Aviv University, Tel-Aviv,
Israel\\
\printead{e1,e2}} \runauthor{Shay Kels and Nira Dyn}
\thankstext{T1}{This is a preprint version of the paper.}
\thankstext{t1}{corresponding author}
\date{}
\maketitle
\begin{abstract}

In this work we construct subdivision schemes refining general
subsets of $\mathbb{R}^n$ and study their applications to the
approximation of set-valued functions. Differently from previous
works on set-valued approximation, our methods are developed and
analyzed in the metric space of Lebesgue measurable sets endowed
with the symmetric difference metric. The construction of the
set-valued subdivision schemes is based on a new weighted average of
two sets, which is defined for positive weights (corresponding to
interpolation) and also when one weight is negative (corresponding
to extrapolation).

Using the new average with positive weights, we adapt to sets spline
subdivision schemes computed by the Lane-Riesenfeld algorithm, which
requires only averages of pairs of numbers. The averages of numbers
are then replaced by the new averages of pairs of sets. Among other
features of the resulting set-valued subdivision schemes, we prove
their monotonicity preservation property. Using the new weighted
average of sets with both positive and negative weights, we adapt to
sets the 4-point interpolatory subdivision scheme. Finally we
discuss the extension of the results obtained in metric spaces of
sets, to general metric spaces endowed with an averaging operation
satisfying certain properties.
\end{abstract}

\section{Introduction}
Approximation of set-valued functions (SVFs) has various potential
applications in optimization, control theory, mathematical
economics, medical imaging and more. The problem is closely related
to the approximation of an $N$-dimensional object from a sequence of
its parallel cross-sections, since such an object can be regarded as
a univariate set-valued function with sets of dimension $N-1$ as
images \cite{dyn2002spline,levin1986multidimensional}. In particular
for $N =3$, the problem is important in medical imaging and is known
as "reconstruction from parallel cross-sections" (\emph{see
e.g.}\cite{albu2008morphology,barequet2007nonlinear,bors2002binary}
\emph{and references therein}).

Motivated by the problem of set-valued approximation and its
applications, we develop and study set-valued subdivision schemes.
Real-valued subdivision schemes repeatedly refine numbers and
generate limit functions. When applied componentwise to points in
$\mathbb{R}^3$, the schemes generate smooth curves/surfaces, and as
such, are widely used in Computer Graphics and Geometric Design.
When the initial data are samples of a function, the limit of the
subdivision approximates the sampled function. For a general review
on subdivision schemes see \cite{dyn2003subdivision}. In this work,
we propose a new method for the adaptation of subdivision schemes to
sets, and show convergence and approximation properties of the
resulting set-valued subdivision schemes.

In the case of data consisting of convex sets, methods based on the
classical \emph{Minkowski sum} of sets can be used
\cite{dyn2000spline,vitale1979approximation}. In this approach, sums
of numbers in positive operators for real-valued approximation are
replaced by Minkowski sums of sets. A more recent approach is to
embed the given convex sets into the Banach space of \emph{directed
sets} \cite{baier2001differences}, and to apply any existing method
for approximation in Banach spaces \cite{baier2011set}.

The case of data consisting of general sets (not necessarily
convex), which is relevant in many applications, is more
challenging. For data sampled from a set-valued function with
general sets as images, methods based on Minkowski sum of sets fail
to approximate the sampled function
\cite{dyn2005set,vitale1979approximation}. So other operations
between sets are needed.

In the spirit of Frechet expectation \cite{frechet1948elements}, Z.
Artstein proposed in \cite{artstein1989pla} to interpolate data
sampled from a set-valued function in a piecewise way, so that for
two sets $F_1,F_2 \subset \mathbb{R}^n$ given at consecutive points
$x_1,x_2$,  the interpolant $F\left(\cdot\right)$ satisfies for any
$t_1,t_2 \in \left[x_1,x_2\right]$,
\begin{equation}\label{MetricProperty1}
{dist}\left( {F\left( {t_1 } \right),F\left( {t_2 } \right)} \right)
= \frac{{\left| {t_2  - t_1 } \right|}} {{x_2  - x_1 }}{dist}\left(
{F_1 ,F_2 } \right) ,
\end{equation}
where  $dist$ is a metric on sets. The relation
(\ref{MetricProperty1}) is termed the \emph{metric property}. A
weighted average of two sets introduced in \cite{artstein1989pla},
and termed later as the \emph{metric average}, leads to a piecewise
interpolant satisfying the metric property relative to the Hausdorff
metric.

Extending this work, in \cite{dyn2001spline} the metric average is
applied in the Lane-Riesenfeld algorithm for spline subdivision
schemes. The set-valued subdivision schemes obtained this way are
shown to approximate SVFs with general sets as images. The
convergence and the approximation results obtained in
\cite{dyn2001spline} are based on the metric property of the metric
average. The adaptation to sets of certain positive linear operators
based on the metric average is described in
\cite{dyn2006approximations}. For reviews on set-valued
approximation methods see also \cite{dynapproximation,muresanset}.

As it is noticed in \cite{artstein1989pla}, the particular choice of
metric is crucial to the construction and analysis of set-valued
methods. While previous works develop and analyze set-valued
approximation methods in the metric space of compact sets endowed
with the Hausdorff metric, we consider the problem in the metric
space of Lebesgue measurable sets with the symmetric difference
metric\footnote{The measure of the symmetric difference is only a
pseudo-metric on Lebesgue measurable sets. The metric space is
obtained in a standard way as described in Section
\ref{sectionPreliminaries}}. Our setting allows us to approximate
SVFs, which are H{\"o}lder continuous in the symmetric difference
metric but may be discontinuous in the Hausdorff metric, as
illustrated by the following simple example.
\begin{example}\label{ExampMotivation}
Let $F$ be a SVF from $\mathbb{R}$ to subsets of $\mathbb{R}$,
\begin{equation}
F\left( x \right) = \left\{ {y:1 \leq y \leq 2 - \left| x \right|}
\right\} \bigcup \left\{ {y:3 \leq y \leq 4 - 2\left| x \right|}
\right\} \ ,
\end{equation}
with graph given in Figure \ref{fig:Discont}. It is easy to observe
that $F$ is discontinues at $x = \frac{1}{2}$ (and also at $x =
-1,-\frac{1}{2},1$), if the distance between subsets of $\mathbb{R}$
is measured in the Hausdorff metric, but it is Lipschitz continuous
everywhere if the distance is measured in the symmetric difference
metric.
\end{example}
\begin{figure}
\begin{center}
\includegraphics[scale=0.3]{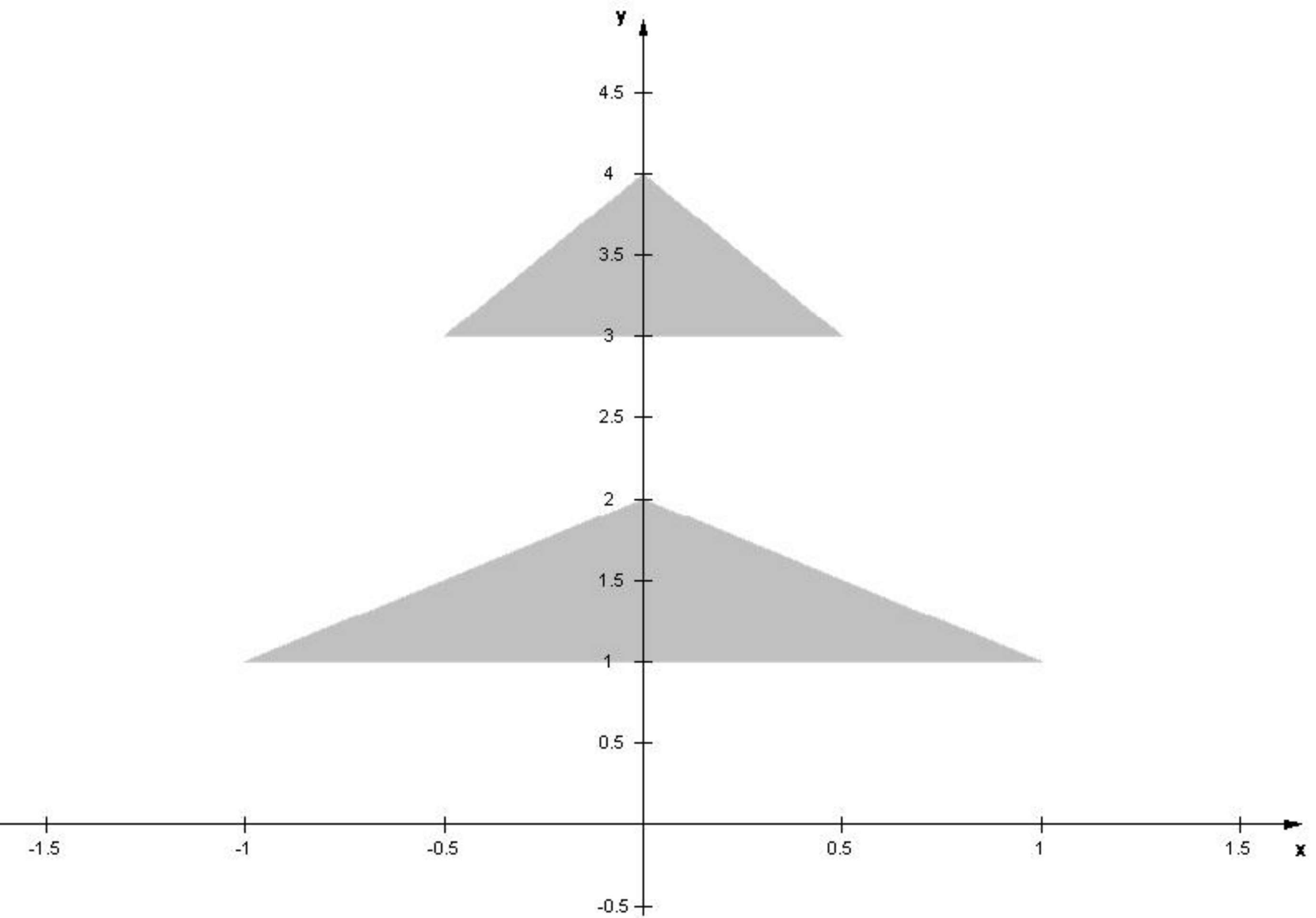}
\end{center}
\caption{The graph of a SVF $F\left(x\right)$, which is Lipschitz
continuous relative to the symmetric difference metric, while
discontinuous relative to the Hausdorff metric.} \label{fig:Discont}
\end{figure}
On the other hand, under mild assumptions on the sets
$F\left(x\right)$, H{\"o}lder continuity in the Hausdorff metric
implies H{\"o}lder continuity in the symmetric difference metric
\cite{galton2000qualitative}. Since there is an intrinsic connection
between the continuity and approximability of a function,
approximation results for H{\"o}lder continuous SVFs obtained in the
symmetric difference metric apply to a wider class of functions than
similar results obtained in the Hausdorff metric.

In order to develop approximation methods in the space of sets with
the symmetric difference metric, we introduce a new binary weighted
average of \emph{regular compact Jordan measurable} sets. The new
average is built upon the method introduced in
\cite{levin1986multidimensional}, which is known as the
\emph{shaped-based interpolation} in the engineering literature
\cite{herman1992shape,raya1990shape}. Our new average has the metric
property relative to the symmetric difference metric. In addition
when both weights are non-negative, the measure of the average of
the two sets is equal to the average with the same weights of the
measures of the two sets.  We term this feature of the new average
the \emph{measure property}, and term our average the \emph{measure
average}.

The measure average performs locally on each \emph{connected
component} of the symmetric difference of the two operand sets,
leading to satisfactory geometric performance, which is essential in
many applications. In particular, the ideas of this work lead to a
practical algorithm for the reconstruction of 3D objects from their
2D cross-sections described in \cite{kels2011reconstruction}.

First we use the measure average to interpolate between a sequence
of sets in a piecewise way. Then we adapt to sets spline subdivision
schemes, expressed in terms of repeated binary averages of numbers
using the Lane-Riesenfeld algorithm \cite{lane1980theoretical}. As
in the case of the metric average \cite{dyn2002spline}, we prove
convergence of the spline subdivision to a Lipschitz continuous
limit SVF $F^\infty  \left( \cdot \right)$. It follows from the
measure property of the measure average, that $\mu \left( {F^\infty
\left( \cdot \right)} \right)$ is the limit of the same spline
subdivision scheme applied to the measures of the initial sets.
Moreover, we prove that spline subdivision schemes adapted to sets
with the measure average are monotonicity preserving in the sense of
the set-inclusion relation.

It is well known, that in order to obtain "reasonable" interpolation
methods other than the piecewise interpolation, some notion of
extrapolation is needed. Our measure average of sets is defined for
positive weights and also when one weight is negative, therefore it
performs both interpolation and extrapolation. Using the measure
average with both negative and non-negative weights, we adapt to
sets the 4-point interpolatory subdivision scheme of
\cite{dyn19874}. This is the first adaptation of an interpolatory
subdivision scheme to sets. We prove that the 4-point subdivision
scheme based on the measure average converges to a continuous limit
SVF and approximates H{\"o}lder continuous SVFs, when the initial
sets are samples of such a function.

We observe that many results on set-valued subdivision obtained in
this and previous works are based on the triangle inequality in the
underlying metric space, along with the metric property of the
average of sets. Employing this observation, we extend several
results obtained in the context of metric spaces of sets, to general
metric spaces endowed with an average having the metric property.

The structure of this work is as follows. Preliminary definitions
are given in Section \ref{sectionPreliminaries}. In Section
{\ref{sectionDistanceAverage}}, we study properties of the average
of sets based on the method in \cite{levin1986multidimensional} ,
which are relevant to the construction of our measure average. In
Section \ref{sectionMeasureAverage} we introduce our measure average
of sets, prove its important features and apply it to the
interpolation between sets in a piecewise way. Spline subdivision
schemes based on the measure average are studied in Section
\ref{sectionSplineSubdivision}, while in Section
\ref{section4pointScheme}, we adapt to sets the 4-point subdivision
scheme. In Section \ref{sectionComputation}, we provide several
computational examples illustrating our analytical results. Finally,
extensions of some of the results obtained in metric spaces of sets
to general metric spaces are discussed in Section
\ref{sectionExtensions}.

\section{Preliminaries}\label{sectionPreliminaries}
First we introduce some definitions and notation. We denote by $\mu$
the n-dimensional \emph{Lebesgue measure} and by $\mathfrak{L}_n$
the collection of \emph{Lebesgue measurable subsets} of
$\mathbb{R}^n$ having finite measure. The \emph{set difference} of
two sets $A,B$ is,
\begin{equation*}
A \setminus B = \left\{ {p:p \in A,p \notin B} \right\} \ ,
\end{equation*}
and the \emph{symmetric difference} is defined by,
\begin{equation*}
A\Delta B = A\setminus B \bigcup B\setminus A \ .
\end{equation*}
 The \emph{measure of the symmetric difference} of $A,B \in \mathfrak{L}_n$,
\begin{equation*}
d_\mu  \left( {A,B} \right) = \mu \left( {A\Delta B} \right) \ ,
\end{equation*}
induces a pseudo-metric on $\mathfrak{L}_n$, and  $\left(
{\mathfrak{L}_n ,d_\mu  } \right)$ is a complete metric space by
regarding any two sets $A, B$ such that $\mu \left( {A\Delta B}
\right) = 0$ as equal (\cite{halmos1974measure}, Chapter 8). For
$A,B \in \mathfrak{L}_n$, such that $B \subseteq A$, it is easy to
observe that
\begin{equation}\label{IncludedDistance}
d_\mu  \left( {A,B} \right) = \mu \left( {A\setminus B} \right) =
\mu \left( A \right) - \mu \left( B \right) \ .
\end{equation}

The \emph{boundary} of a set $A$ is denoted by $\partial A$, and we
use the notation ${\operatorname{ci}}\left( A \right)$ for the
\emph{closure} of the \emph{interior} of $A$. A bounded set $A$,
such that $A = \operatorname{ci}\left(A\right)$ is called
\emph{regular compact}. Regular compact sets are closed under finite
unions, but not under finite intersections, yet for $A,B$ regular
compact such that $B \subset A$, it holds trivially that $A \cap B =
B = \operatorname{ci}\left( {A \cap B} \right)$.

We recall that a set $A$ is \emph{Jordan measurable} if and only if
$ \mu \left( {\partial A} \right) = 0$. It is easy to see that for a
Jordan measurable $A$,
\begin{equation}\label{jordanCIA}
\mu \left( A \right) = \mu \left( {{\operatorname{ci}}\left( A
\right)} \right) \ .
\end{equation}

We denote by $\mathfrak{J}_n$ the subset of $\mathfrak{L}_n$
consisting of \emph{regular compact Jordan measurable sets}. Notice
that for any $A,B \in \mathfrak{J}_n$, $d_\mu  \left( {A,B} \right)
= 0$ implies $A = B$, therefore $d_\mu$ is a metric of
$\mathfrak{J}_n$. Moreover, by its definition $\mathfrak{J}_n$ is
closed under finite unions.

We recall that a set $A$ is called \emph{connected} if there are no
two disjoint open sets $V_1,V_2 \subset \mathbb{R}^n$, such that $A
= \left( {A \bigcap V_1 } \right) \bigcup \left( {A \bigcap V_2 }
\right)$. The set $C \subseteq D$ is called a \emph{connected
component} of $D$ if it is connected, and if there is no connected
set $B$, such that $C \subset B \subseteq D$.

\section{The "distance average" of sets}\label{sectionDistanceAverage}
The basic tool for the construction of the measure average of sets
to be introduced in the next section is what we call the
\emph{distance average} of sets. In this section we derive
properties of the distance average that are relevant to our
construction.

\subsection{Definition and basic properties}
The distance average is based on the method introduced in
\cite{levin1986multidimensional}, which employs the \emph{signed
distance functions} of sets. The \emph{signed distance} from a point
$p$ to a non-empty set $A \subset \mathbb{R}^n$ is defined by,
\begin{equation}
d_S\left( {p,A} \right) = \left\{ \begin{gathered}
  d\left( {p, \partial A} \right)\ \ \ \ \ \ \ \ p \in A \\
   - d\left( {p, \partial A} \right)\ \ \ \ \ \ \ p \notin A \ ,\\
\end{gathered}  \right .
\end{equation}
where $d\left(q,B\right)$ is the Euclidean distance from a point $q$
to a set $B$, namely
\[
d\left( {q,B} \right) = \mathop {\min }\limits_{b \in B} \left\| {q
- b} \right\| \ .
\]
The \emph{signed distance function} of $A$ is defined on
$\mathbb{R}^n$ as $d_S\left( {\cdot,A} \right)$.

\begin{definition}\label{DistanceAverageDef}
The distance average with the averaging parameter $x \in \mathbb{R}$
of two not-empty sets $A,B \in \mathfrak{J}_n$ is,
\begin{equation}\label{AuxiliaryAverage}
xA\widetilde  \bigoplus \left( {1 - x} \right)B = {\left\{
{p:f_{A,B,x} (p) \geq 0} \right\}}  \ ,
\end{equation}
where
\begin{equation}\label{fAB}
f_{A,B,x} \left( p \right) = xd_S \left( {p,A} \right) + (1 - x)d_S
\left( {p,B} \right) \ .
\end{equation}
\end{definition}
Note that $f_{A,B,x}$ is not the signed distance function of the set
$xA\widetilde  \bigoplus \left( {1 - x} \right)B$. Also note that
$f_{A,B,x}$ is continuous by the continuity of the distance
function.

We observe a few properties of the distance average that are
relevant to the construction of the measure average in the next
section.
\begin{lemma}\label{AuxiliaryAverageProperties} Let $A,B \in \mathfrak{J}_n$ and $x \in \mathbb{R}$, then
\begin{enumerate}
\item \label{AuxiliaryThroughTheEnds}
$0A \widetilde \bigoplus 1B = B$, $1A \widetilde \bigoplus 0B = A$

\item \label{AuxiliaryAwithA}
$xA \widetilde \bigoplus (1-x) A \mathop  = A$

\item\label{AuxiliaryInclusionProperty}
For $B \subseteq A$, $x_1 \leq x_2$,  $x_1A \widetilde \bigoplus
\left( {1 - x_1} \right)B \subseteq x_2A\widetilde \bigoplus \left(
{1 - x_2} \right)B $

\item \label{AuxiliaryMiddleProperty} For $x \in [0,1]$, $A \bigcap B \subseteq xA\widetilde  \bigoplus \left( {1 - x} \right)B \subseteq A \bigcup B$

\item \label{AuxiliaryBoundedProperty} $xA \widetilde \bigoplus \left(1 - x\right)B$ is a bounded closed set

\end{enumerate}
\end{lemma}
\begin{proof}
Properties \ref{AuxiliaryThroughTheEnds}-\ref{AuxiliaryAwithA}
follow from Definition \ref{DistanceAverageDef}. To obtain Property
\ref{AuxiliaryInclusionProperty}, observe that $d_S(p,B) \leq
d_S(p,A)$ and consequently $ f_{A,B,x_1} \left( p \right) \leq
f_{A,B,x_2} \left(p \right) $.

To prove Property \ref{AuxiliaryMiddleProperty}, note that for $p
\in A \bigcap B$, $d_S \left( {p,A} \right) \geq 0,d_S \left( {p,B}
\right) \geq 0$, so for $x \in [0,1]$, $f_{A,B,x}\geq 0$ and $p \in
xA \widetilde \bigoplus (1-x)B$. Now with a similar argument for $p
\notin A \bigcup B$, Property \ref{AuxiliaryMiddleProperty} is
proved.

The set $xA \widetilde \bigoplus \left(1 - x\right)B$ is closed by
definition, so in order to prove Property
\ref{AuxiliaryBoundedProperty}, it is sufficient to show that for
any $x$ this set is bounded. Since $A,B$ are bounded, so is their
union. Therefore for $x \in [0,1]$, Property
\ref{AuxiliaryBoundedProperty} follows from Property
\ref{AuxiliaryMiddleProperty}. For $x \notin [0,1]$, in view of
Definition \ref{DistanceAverageDef} we have,
\begin{equation*}
 xA \widetilde  \bigoplus \left( {1 - x} \right)B \subseteq  {A \bigcup B} \bigcup \zeta _{A,B,x} \ ,
\end{equation*}
where
\begin{equation}\label{zetaSet}
\zeta _{A,B,x}  = \left\{ {p:d_S \left( {p,A} \right) < 0,d_S \left(
{p,B} \right) < 0,xd_S \left( {p,A} \right) + \left( {1 - x}
\right)d_S \left( {p,B} \right) \geq 0} \right\} \ .
\end{equation}
Since $A \bigcup B$ is bounded, it is enough to show that also
$\zeta _{A,B,x}$ is bounded. Without loss of generality, assume that
$x > 1$. There exists $\theta  > 0$, such that for any $p \in
\mathbb{R}^n$,
\begin{equation}\label{theta}
\left| {d_S \left( {p,A} \right) - d_S \left( {p,B} \right)} \right|
< \theta \ .
\end{equation}
From (\ref{zetaSet}) and (\ref{theta}),
\begin{equation}\label{zetaSet1}
\zeta _{A,B,x}  \subseteq \left\{ {p:d_S \left( {p,A} \right) <
0,d_S \left( {p,B} \right) < 0,xd_S \left( {p,A} \right) + \left( {1
- x} \right)\left( {d_S \left( {p,A} \right) - \theta } \right) \geq
0} \right\} \ .
\end{equation}
and therefore,
\begin{equation*}
\zeta _{A,B,x}  \subseteq \left\{ {p:d_S \left( {p,A} \right) <
0,d_S \left( {p,B} \right) < 0,\left| {d_S \left( {p,A} \right)}
\right| \leq \left( {x - 1} \right)\theta } \right\} \ ,
\end{equation*}
from which it follows that $\zeta _{A,B,x}$ is bounded. \qed
\end{proof}

We extend the domain of the distance average to include the empty
set. In case $B = \phi$, $A \neq \phi$, choose a "center" point $q$
of $A$ such that,
\begin{equation}\label{SupDist}
d_S \left( {q,A} \right) = \mathop {\sup }\limits_{a \in A} \left\{
{d_S \left( {a,A} \right)} \right\} \ .
\end{equation}
We define $xA \widetilde \bigoplus (1-x)B$ as the set,
\begin{equation*}
\left\{ {p:xd_S \left( {p,A} \right) + \left( {x - 1} \right)\left\|
{p - q} \right\| \geq 0} \right\} \ .
\end{equation*}
The average of two empty sets is the empty set. One can verify that
with these definitions Properties
\ref{AuxiliaryThroughTheEnds}-\ref{AuxiliaryBoundedProperty} of
Lemma \ref{AuxiliaryAverageProperties} are preserved.

The distance average does not satisfy the metric and the measure
properties. For example, for any non-empty $A,B \in \mathfrak{J}_n$
that are \emph{disjoint} ,  $\frac {1}{2}A \widetilde \bigoplus
\frac{1}{2}B$ is the empty set. Also consider the distance average
$\frac{1} {2}A \widetilde \bigoplus \frac{1} {2}B$ with  $A = [0,3]$
and $B=[0,1]\bigcup[2,3]$. From Property
\ref{AuxiliaryMiddleProperty} of Lemma
\ref{AuxiliaryAverageProperties}, we have that,$\frac{1} {2}A
\widetilde \bigoplus \frac{1} {2}B \subseteq A$, on the other hand
we now show the opposite inclusion. Let $p \in A \setminus B =
\left( {1,2} \right)$, it is easy to observe that $d_S \left( {p,A}
\right) > 0$, $d_S \left( {p,B }\right) < 0$ and $\left| {d_S \left(
{p,A} \right)} \right| > \left| {d_S \left( {p,B} \right)} \right|$.
Consequently, from the definition of the distance average $p \in
\frac{1} {2}A \widetilde \bigoplus \frac{1} {2}B$, thus $\frac{1}
{2}A \widetilde \bigoplus \frac{1} {2}B = A$. Of course this average
does not satisfy the metric property or the measure property.
Moreover, it is undesirable to obtain one of the original sets, as
an equally weighted average of the two different sets, since such
average does not reflect a gradual transition between the two sets.
We aim to define a new set average with desirable properties, using
the distance average as the basic tool for the construction.

Before presenting the new set average, we consider the measure of
the distance average as a function of the averaging parameter,
\begin{equation}\label{hFunction}
h(x) = \mu \left( {xA \widetilde  \bigoplus \left( {1 - x} \right)B}
\right) \ , \ x \in \mathbb{R} \ ,
\end{equation}
and study conditions for its continuity. Note that by Property
\ref{AuxiliaryBoundedProperty} in Lemma
\ref{AuxiliaryAverageProperties}, $h\left(x\right)$ is well defined
for all $x \in \mathbb{R}$.

\subsection{Continuity of the measure of the distance average}
First we prove a result, which might be of interest beyond its
application in our context,
\begin{lemma}\label{EuclideanDistanceMeasure}
Let $A,B$ be closed sets, and let $\lambda > 0, \lambda \neq 1$.
Define
\begin{equation}\label{mLambdaSet}
M_{A,B,\lambda }  = \left\{ {p:p \notin A \bigcup B,d\left( {p,A}
\right) = \lambda d\left( {p,B} \right)} \right\}\ ,
\end{equation}
then $\mu \left( {M_{A,B,\lambda}  } \right) = 0$.
\end{lemma}
\begin{proof}
To prove the claim of the lemma, it is sufficient to show that for
any $p \in M_{A,B,\lambda}$, there exists a cone of constant angle
with $p$ as its vertex, which is not in  $M_{A,B,\lambda}$. Without
loss of generality assume that $\lambda > 1$. Let $p \in
M_{A,B,\lambda}$, then there exists an open ball of radius $d(p, A)
= \lambda d(p, B)$ around $p$ that contains no points of $A$, and
another open ball around $p$, of radius $d(p, B)$, which contains no
points of $B$. There is at least one point $v \in B$, such that
$\left\| {p - v} \right\| = d \left( {p,B} \right)$. Let
$\varepsilon \in \left( {0,d \left( {p,B} \right)} \right)$ and
$x_\varepsilon$ be the point on the segment $[p,v]$, such that
$\left\| {x_\varepsilon   - v} \right\| = \varepsilon$. The open
ball of radius $\lambda d(p, B) - \varepsilon$ around
$x_\varepsilon$ contains no points of $A$, and $d \left(
{x_\varepsilon, B} \right) = d \left( {p, B} \right) - \varepsilon
$. Since $\lambda \left( {d \left( {p,B} \right) - \varepsilon }
\right) < \lambda d \left( {p,B} \right) - \varepsilon $, there is
no point of $A$ at distance $\lambda d \left( {x_\varepsilon, B}
\right)$ from $x_\varepsilon$, and therefore  $x_\varepsilon \notin
M_{A,B,\lambda}$. Consider now a point $x'_{\varepsilon}$ at
distance $r$ from $x_\varepsilon$. By the triangle inequality, $d
\left( {x'_{\varepsilon}, B} \right) \leq d \left( {x_{\varepsilon},
B} \right) + r$ and $d \left( {x'_{\varepsilon}, A} \right) \geq d
\left( {x_{\varepsilon},A} \right) - r$. Therefore,
\begin{equation}\label{distXprimeB2}
d \left( {x'_{\varepsilon}, B} \right) \leq d \left( {p, B} \right)
- \varepsilon  + r \ ,
\end{equation}
and
\begin{equation}\label{distXprimeA2}
d \left( {x'_{\varepsilon}, A} \right) \geq \lambda d \left( {p, B}
\right) - \varepsilon  - r \ .
\end{equation}
In order to obtain that  $x'_{\varepsilon} \notin M_{A,B,\lambda}$,
it is enough to show that,
\begin{equation}
\lambda d \left( {x'_{\varepsilon}, B} \right) < d \left(
{x'_{\varepsilon}, A} \right) \ ,
\end{equation}
or using (\ref{distXprimeB2}) and (\ref{distXprimeA2}),
\begin{equation}\label{rEquation}
\lambda d \left( {p, B} \right) - \lambda \varepsilon  + \lambda r <
\lambda d \left( {p, B} \right) - \varepsilon  - r \ .
\end{equation}
Solving (\ref{rEquation}) for $r$, we obtain,
\begin{equation}
r < \frac{{\lambda  - 1}} {{\lambda  + 1}}\varepsilon  \ .
\end{equation}
So the open ball $B_\varepsilon$ of radius $\frac{{ {\lambda  - 1}
}} {{ {\lambda + 1 } }}\varepsilon$ around $x_{\varepsilon}$ has
empty intersection with the set $M_{A,B,\lambda}$. Let $U_p$ be the
union of the balls $B_\varepsilon$ for $\varepsilon \in  \left( {0,d
\left( {p, B} \right)} \right)$,
\begin{equation}
U_p = \bigcup\limits_{\varepsilon  \in (0,d(p, B))} {B_\varepsilon}
\ .
\end{equation}
Any point $q$ in the open ball of radius $d(p, B)$ around $p$, such
that the angle $\angle qpv$ satisfies,
\begin{equation}
0 \leq \tan \angle qpv \leq \;\frac{{\lambda  - 1}} {{\lambda  + 1}}
\ ,
\end{equation}
is in $U_p$. We conclude that for any $p \in M$ there is a cone
$U_p$ of a constant angle based at $p$, such that $U \bigcap
M_{A,B,\lambda} = \phi$. In view of the \emph{Lebesgue density
theorem} (\emph{see. e.g.} \cite{mattila1999geometry}, Corollary
2.14), we obtain that $M_{A,B,\lambda}$ has zero Lebesgue measure.
\qed
\end{proof}
Note that it is easy to construct closed sets $A,B$ with non-empty
intersection, such that the claim of Lemma
\ref{EuclideanDistanceMeasure} does not hold for $\lambda = 1$.

From Lemma \ref{EuclideanDistanceMeasure} we conclude,
\begin{corollary}\label{ZeroLevelSetMeasure}
Let $A,B \in \mathfrak{J}_n$ and $t \in \mathbb{R}$. Define the set
$\Omega_{A,B,x}$ as,
\begin{equation}\label{OmegaSet}
\Omega_{A,B,x} = \left\{ {p:f_{A,B,x} \left( p \right) = 0} \right\}
.
\end{equation}
Then for $x \neq \frac{1}{2}$,
\begin{equation}\label{eqZeroMeasure}
\mu \left( \Omega_{A,B,x} \right) = 0 \ .
\end{equation}
\end{corollary}
\begin{proof} For $x \ne 0,1$, set $\lambda = \left|\frac{x}{1-x}\right|$, then
\begin{equation*}
\Omega _{A,B,x}  \subseteq M_{\partial A,\partial B,\lambda }
\bigcup \partial A \bigcup \partial B \ ,
\end{equation*}
with $M_{\partial A,\partial B,\lambda}$ defined by
(\ref{mLambdaSet}). It follows from Lemma
\ref{EuclideanDistanceMeasure} and the assumptions $A,B \in
\mathfrak{J}_n$ and $x \neq \frac{1}{2}$, that $\mu \left( {\Omega
_{A,B,x} } \right) = 0$. For $x = 0$ or $x = 1$, $\Omega_{A,B,x}$
equals to $\partial B$ or $\partial A$ respectively, so $\mu \left(
\Omega_{A,B,x} \right) = 0$, by the assumption that $A,B \in
\mathfrak{J}_n$. \qed
\end{proof}
Corollary \ref{ZeroLevelSetMeasure}, does not treat the case $x =
\frac {1}{2}$. Indeed, it is not difficult to give two sets $A,B \in
\mathfrak{J}_n$ such that $\mu \left( \Omega_{A,B,\frac{1}{2}}
\right) > 0$, see Figure \ref{fig:Triangle} for an example of such
two sets.

Since for the continuity of $h(x)$ at $x = \frac{1}{2}$, we need the
condition $\mu \left( \Omega_{A,B,\frac{1}{2}} \right) = 0$, we
introduce the following relation between two sets.
\begin{definition}
$A,B \in \mathfrak{J}_n$ satisfy the zero-measure condition if $\mu
\left( \Omega_{A,B,\frac{1}{2}} \right) = 0$.
\end{definition}
\begin{remark}\label{RemarkZeroMeasureCondtion}
In view of Corollary \ref{ZeroLevelSetMeasure}, if $A,B$ satisfy the
zero-measure condition, then
\begin{equation}
\mu \left( \Omega_{A,B,x} \right) = 0 \ ,
\end{equation}
for any $x$.
\end{remark}

\begin{remark}\label{Nongenerity}
For a Lebesgue measurable function $f : \mathbb{R}^n \to
\mathbb{R}$, the set of values  $c$ such that,
\begin{equation*}
\mu \left( {\left\{ {p:f\left( p \right) = c} \right\}} \right) > 0
\ ,
\end{equation*}
has zero measure in view of Fubini's theorem  (\emph{see e.g.}
\cite{halmos1974measure}, Section 36). Consequently if the level set
$\left\{ {p : f_{A,B,\frac{1} {2}} \left( p \right) = 0 } \right\}$
has non-zero measure, one can always choose an arbitrarily small
$\varepsilon > 0$ such that, $\mu \left( {\left\{ {p :
f_{A,B,\frac{1} {2}} \left( p \right) = \varepsilon } \right\}}
\right) = 0 $. Therefore, the case of $A,B$ satisfying the
zero-measure condition is generic, while the case that this
condition is not satisfied is degenerate.
\end{remark}

Next we adapt to our context a basic result from probability theory,
\begin{lemma}\label{FunctionContinuityLemma}
Let $A,B \in \mathfrak{J}_n$, $B \subset A$ and $\mu \left( {\Omega
_{A,B,x^* } } \right) = 0$, then the function $h(x)$ defined by
(\ref{hFunction}) is continuous at $x^*$.
\end{lemma}
\begin{proof}
To see that $h(x)$ is left-continuous at $x^*$, let $x_n \rightarrow
x^*$, $x_{n} \leq x_{n+1}$. By Property
\ref{AuxiliaryInclusionProperty} of Lemma
\ref{AuxiliaryAverageProperties},
\begin{equation*}
x_n A \widetilde \bigoplus \left( {1 - x_n } \right)B \subseteq
x_{n+1} A \widetilde \bigoplus \left( {1 - x_{n+1} } \right)B \ ,
\end{equation*}
so
\begin{equation*}
x^* A{\text{ }}\widetilde \bigoplus \left( {1 - x^* } \right)B =
\left( {\bigcup\limits_{n = 1}^\infty  {x_n A \widetilde\bigoplus
\left( {1 - x_n } \right)B} } \right) \bigcup{\Omega _{A,B,x^* } } \
.
\end{equation*}
Consequently by the continuity from below of the Lebesgue measure
(\emph{see e.g.} \cite{jones2001lebesgue}, Chapter 2, Section B, M5)
and by the assumption $\mu \left( {\Omega _{A,B,x^* } } \right) =
0$,
\begin{equation*}
h\left( {x^* } \right) = \mathop {\lim }\limits_{n \to \infty }
h\left( {x_n } \right) \ .
\end{equation*}
To obtain that $h(x)$ is right-continuous at $x^*$, let $x_n
\rightarrow x$, $x_{n} \geq x_{n+1}$. Then in view of Property
\ref{AuxiliaryInclusionProperty} of Lemma
\ref{AuxiliaryAverageProperties},
\begin{equation*}
x^* A \widetilde  \bigoplus \left( {1 - x^* } \right)B =
{\bigcap\limits_{n = 1}^\infty  {x_n A{\text{ }}\widetilde
\bigoplus \left( {1 - x_n } \right)B} } \ .
\end{equation*}
Using Property \ref{AuxiliaryBoundedProperty} of Lemma
\ref{AuxiliaryAverageProperties}, by the continuity from above of
the Lebesgue measure (\emph{see e.g.} \cite{jones2001lebesgue},
Chapter 2, Section B, M6)  we obtain that $h(x)$ is right-continuous
at $x^*$. \qed
\end{proof}

\begin{figure}
\begin{center}
\includegraphics[scale=0.5]{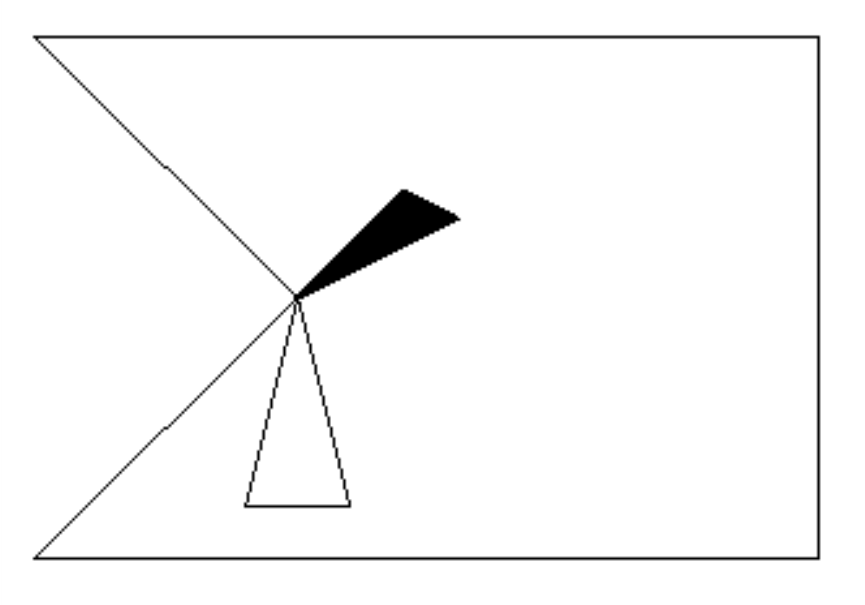}
\end{center}
\caption{The pentagon and the lower triangle represent the
boundaries of the sets $A$ and $B$ respectively. The black set is
contained in $\Omega _{A,B,\frac{1}{2}} $.} \label{fig:Triangle}
\end{figure}
In view of Corollary \ref{ZeroLevelSetMeasure}, by Remark
\ref{RemarkZeroMeasureCondtion} and Lemma
{\ref{FunctionContinuityLemma}}, we arrive at
\begin{corollary}\label{FunctionContinuity}
If $A,B \in \mathfrak{J}_n$ satisfy the zero-measure condition, then
$h(x)$ is continuous everywhere. Otherwise $h(x)$ is continuous at
all $x$ except at $x = \frac{1}{2}$.
\end{corollary}
Finally, we discuss the Jordan measurability of $xA \widetilde
\bigoplus (1- x) B$,
\begin{corollary}\label{AuxiliaryJordan} If $A,B \in \mathfrak{J}_n$ satisfy the zero-measure condition, then $xA \widetilde \bigoplus (1- x) B$ is Jordan measurable for all $x$.
Otherwise it is Jordan measurable for all $x$ except may be for $x =
\frac{1}{2}$.
\end{corollary}
\begin{proof}
By Definition \ref{DistanceAverageDef},
\begin{equation}
 {\partial \left( {xA \widetilde\bigoplus \left( {x - t} \right)B} \right)}  \subseteq  \Omega _{A,B,x}  \ ,
\end{equation}
which in view of Corollary \ref{ZeroLevelSetMeasure} and Remark
\ref{RemarkZeroMeasureCondtion}, leads to the claim of the
corollary. \qed
\end{proof}
However the distance average $tA \widetilde \bigoplus \left(1 -
t\right)B$ is not necessarily regular compact, even when $A$ and $B$
are.

\section{Construction of the "measure average" of sets}\label{sectionMeasureAverage}
Our construction aims to achieve several important properties of the
measure average, denoted by $tA \bigoplus \left( {1 - t} \right)B$.
We use these properties later on in the analysis of the subdivision
methods based on the measure average. For $A,B \in \mathfrak{J}_n$
and an averaging parameter $t \in \mathbb{R}$, the desired
properties of the measure average are
\begin{listprop}\label{PropertiesList}\textbf{}
\begin{enumerate}
\item \label{PropClousure} $tA \bigoplus \left( {1 - t} \right)B \in \mathfrak{J}_n$ (closure property)
\item \label{PropThroughTheEnds} $0A \bigoplus 1B  =  B$, $1A  \bigoplus 0B  =  A$ (interpolation property)
\item \label{PropSubMetric} $d_\mu  \left( {sA \bigoplus \left( {1 - s} \right)B,tA \bigoplus \left( {1 - t} \right)B} \right) \leq \left| {t - s} \right|d_\mu  \left( {A,B} \right)$ (submetric property) \\\\
In addition for $s,t \in \left[0,1\right]$,
\item \label{PropMeasure} $\mu \left( {tA \bigoplus \left( {1 - t} \right)B} \right) = t\mu \left( A \right) + \left( {1 - t} \right)\mu \left( B \right)$ (measure property)
\item \label{PropMetric} $d_\mu  \left( {sA \bigoplus \left( {1 - s} \right)B,tA \bigoplus \left( {1 - t} \right)B} \right) = \left| {t - s} \right|d_\mu  \left( {A,B} \right)$
(metric property)
\item \label{PropIncusion} $\operatorname{ci}\left(A \bigcap B\right) \subseteq tA \bigoplus \left( {1 - t} \right)B\subseteq A \bigcup B $ (inclusion property)
\end{enumerate}
\end{listprop}
The above properties are analogous to those of weighted averages
between non-negative numbers, defined by $\max \left\{ {0,tp +
\left( {1 - t} \right)q} \right\}$, for $p,q \in \mathbb{R}_+$ and
$t \in \mathbb{R}$. In this analogy, the measure of a set, the
measure of the symmetric difference of two sets ($d_\mu  \left( {
\cdot , \cdot } \right)$), and the relation $\subseteq$  are
replaced by the absolute value of a number, the absolute value of
the difference of two numbers and the relation $\leq$ respectively.

The measure average is constructed in three steps, each based on the
previous.
\subsection{The measure average of "simply different" sets}
We begin with the simple case of $A,B \in \mathfrak{J}_n$ such that
$B \subset A$ and $A \setminus B$ consists of only one
\emph{connected component}. We call two such sets \emph{simply
different}. First we assume that the sets $A,B$ satisfy the
zero-measure condition. In this case, the \emph{measure average} is
a reparametrization of the distance average, so that the measure of
the resulting set is as close as possible to the average of the
measures of $A,B$,
\begin{definition}\label{SimpleMeasureAverageSat}
Let $A,B$ be simply different sets satisfying the zero-measure
condition. The measure average of $A,B$ with the averaging parameter
$t \in \mathbb{R}$ is,
\newpage
\begin{equation}\label{SimpDiffMeasureAverage}
tA \bigoplus \left( {1 - t} \right)B = {\operatorname{ci}}\left(
{g\left( t \right)A \widetilde  \bigoplus \left( {1 - g\left( t
\right)} \right)B} \right)
 \ ,
\end{equation}
where $g(t)$ is any parameter in the collection,
\begin{equation}\label{MeasureReparam}
\left\{ {x:x = \mathop {\mathop {\arg \min }\limits_{\xi  \in \left[
{ - N,N} \right]} \left| {h\left( \xi  \right) - \left( {t\mu \left(
A \right) + \left( {1 - t} \right)\mu \left( B \right)} \right)}
\right|}\limits_{} } \right\} \ .
\end{equation}
Here $h$ is defined in (\ref{hFunction}) and $N >> 1$ is a large
positive number.
\end{definition}
The zero-measure condition satisfied by $A,B$ and Corollary
\ref{FunctionContinuity} imply the continuity of $h$. Therefore the
$\arg\min$ in (\ref{MeasureReparam}) is well defined. By Corollary
\ref{AuxiliaryJordan} and the relation (\ref{jordanCIA}),
\begin{equation}
\mu \left( {tA \bigoplus \left( {1 - t} \right)B} \right) = h\left(
{g\left( t \right)} \right) \ .
\end{equation}
The measure average in Definition \ref{SimpleMeasureAverageSat} is a
regular compact set. Since $h(x_1) = h(x_2)$, for any two parameters
$x_1, x_2$ in the collection (\ref{MeasureReparam}), the
corresponding averages defined using either $g(t) = x_1$ or $g(t) =
x_2$ in (\ref{SimpDiffMeasureAverage}) have the same measure. In
addition, by Property \ref{AuxiliaryInclusionProperty} of Lemma
\ref{AuxiliaryAverageProperties}, one of these sets is necessarily
contained in the other. Therefore, in view of
(\ref{IncludedDistance}) and because both sets are regular compact,
\begin{equation*}
x_1 A\widetilde \bigoplus \left( {1 - x_1 } \right)B = x_2
A\widetilde \bigoplus \left( {1 - x_2 } \right)B \ .
\end{equation*}
So any parameter in the collection (\ref{MeasureReparam}) can be
used in (\ref{SimpDiffMeasureAverage}).

\begin{lemma}\label{SimpDiffMeasureAverageProperties}
The measure average of simply different sets satisfying the
zero-measure condition has Properties
\ref{PropClousure}-\ref{PropIncusion} in List of Properties
\ref{PropertiesList}.
\end{lemma}
\begin{proof}
The closure property follows from Definition
\ref{SimpleMeasureAverageSat}. The interpolation property follows
from Definition \ref{SimpleMeasureAverageSat} and from Property
\ref{AuxiliaryThroughTheEnds} in Lemma
\ref{AuxiliaryAverageProperties}.

Notice that since $B \subset A$, Property
\ref{AuxiliaryInclusionProperty} of Lemma
\ref{AuxiliaryAverageProperties} implies that $h(x)$ defined in
(\ref{hFunction}) is monotone non-decreasing. To obtain the
submetric property, we denote $m_{A,B}  = h(-N)$ and  $M_{A,B} =
h(N)$, where  $\left[-N,N\right]$ is the domain used in
(\ref{MeasureReparam}). Since the average $\bigoplus$ is a
reparametrization of the average $\widetilde \bigoplus$,
\begin{equation}\label{SimpleMeasureIsBounded}
\mu \left( {tA \bigoplus \left( {1 - t} \right)B} \right) \in \left[
{m_{A,B} ,M_{A,B} } \right] \ .
\end{equation}
By the continuity of $h$, we obtain from Definition
\ref{SimpleMeasureAverageSat}, that for any $t$ satisfying,
\begin{equation}\label{SimpleMeasureEqualityCondition}
m_{A,B}  \leq t\mu \left( A \right) + \left( {1 - t} \right)\mu
\left( B \right) \leq M_{A,B} \ ,
\end{equation}
we have,
\begin{equation}\label{SimpleMeasureEquality}
\mu \left( {tA \bigoplus \left( {1 - t} \right)B} \right) = t\mu
\left( A \right) + \left( {1 - t} \right)\mu \left( B \right) \ .
\end{equation}
Summarizing (\ref{SimpleMeasureIsBounded}) and
(\ref{SimpleMeasureEquality}), we arrive at
\begin{equation}\label{SimpleMeasureTogether}
\mu \left( {tA \bigoplus \left( {1 - t} \right)B} \right) = \left\{
{\begin{array}{*{20}c}
   \begin{gathered}
  m_{A,B},  \hfill \\
  t\mu \left( A \right) + \left( {1 - t} \right)\mu \left( B \right) \hfill, \\
  M_{A,B},  \hfill \\
\end{gathered}  & \begin{gathered}
  \qquad t\mu \left( A \right) + \left( {1 - t} \right)\mu \left( B \right) \leq m_{A,B}  \hfill \\
  \qquad t\mu \left( A \right) + \left( {1 - t} \right)\mu \left( B \right) \in \left( {m_{A,B} ,M_{A,B} } \right) \hfill \\
  \qquad t\mu \left( A \right) + \left( {1 - t} \right)\mu \left( B \right) \geq M_{A,B}  \hfill \\
\end{gathered}   \\
 \end{array} } \right.
\end{equation}
Assume without loss of generality that $s \leq t$, then from the
monotonicity of $h$ and Definition \ref{SimpleMeasureAverageSat},
$g(s) \leq g(t)$.  By Property \ref{AuxiliaryInclusionProperty} of
Lemma \ref{AuxiliaryAverageProperties} we have,
\begin{equation}\label{SimpDiffInclusionEq}
sA \bigoplus \left( {1 - s} \right)B \subseteq tA \bigoplus \left(
{1 - t} \right)B \ , \ s \leq t \ .
\end{equation}
From (\ref{SimpleMeasureTogether}) we obtain,
\begin{equation*}
\left| {\mu \left( {tA \bigoplus \left( {1 - t} \right)B} \right) -
\mu \left( {sA \bigoplus \left( {1 - s} \right)B} \right)} \right|
\leqslant \left| {t - s} \right|\left( {\mu \left( A \right) - \mu
\left( B \right)} \right) \ ,
\end{equation*}
which in view of (\ref{SimpDiffInclusionEq}) and
(\ref{IncludedDistance}) leads to the proof of the submetric
property.

To prove the measure property, we have to show that for $t \in
[0,1]$ (\ref{SimpleMeasureEqualityCondition}) holds. Observe from
Property \ref{AuxiliaryThroughTheEnds} of Lemma
\ref{AuxiliaryAverageProperties} that,
\begin{equation}
h\left( 0 \right) = \mu \left( B \right),h\left( 1 \right) = \mu
\left( A \right) \ ,
\end{equation}
thus by the monotonicity of $h$,
\begin{equation*}
m_{A,B}  \leq \mu \left( B \right) \leq \mu \left( A \right) \leq
M_{A,B} \ ,
\end{equation*}
and for $t \in [0,1]$,
\begin{equation*}
m_{A,B}  \leq \mu \left( B \right) \leq t\mu \left( A \right) +
\left( {1 - t} \right)\mu \left( B \right) \leq \mu \left( A \right)
\leq M_{A,B} .
\end{equation*}
The metric property follows from the measure property,
(\ref{SimpDiffInclusionEq}) and (\ref{IncludedDistance}). Finally
the inclusion property, follows from the assumption $B \subset A$,
the interpolation property and (\ref{SimpDiffInclusionEq}). \qed
\end{proof}

Next we define the measure average in case of simply different sets
$A,B$ that do not satisfy the zero-measure condition. In view of
Remark \ref{Nongenerity}, this case is non-generic and in
applications can be resolved by a small perturbation of the input
sets. For completeness we provide a formal construction treating
this case.

An $r$-\emph{offset} of a set $B$ with $r \geq 0$, is defined as,
\begin{equation*}
O\left( {B,r} \right) = \left\{ {p:d \left( {p,B} \right) \leq r}
\right\} \ .
\end{equation*}
In case of $A,B$ not satisfying the zero-measure condition, we
intersect the set $A$ with an $r$-\emph{offset} of $B$, where $r$ is
chosen so that the measure of the intersection equals the average of
the measures of $A,B$.
\begin{definition}\label{SimpleMeasureAverageNotSat}
Let $A,B$ be simply different sets that do not satisfy the
zero-measure condition. For $t \in [0,1]$, the measure average of
$A,B$ is defined by,
\begin{equation}\label{SimpleSpecial}
tA \bigoplus \left( {1 - t} \right)B = \operatorname{ci} \left(
O\left( {B,r_{A,B} (t)} \right) \bigcap A \right),
\end{equation}
where $r_{A,B} (t)$ is any number in the collection,
\begin{equation}\label{BallMeasure}
\left\{ {r:\mu \left( {O\left( {B,r} \right)\bigcap A } \right) =
t\mu \left( A \right) + \left( {1 - t} \right)\mu \left( B \right)}
\right\} \ ,
\end{equation}
which is not empty, as is proved in Lemma
\ref{SimpDiffMeasurePropsSpecial}. For $t \notin [0,1]$ the measure
average is defined as in Definition \ref{SimpleMeasureAverageSat}.
\end{definition}
Note that for any $r_1,r_2$ in the collection (\ref{BallMeasure}),
\begin{equation*}
{\operatorname{ci}}\left( {O\left( {B,r_1 } \right)\bigcap A }
\right) = {\operatorname{ci}}\left( {O\left( {B,r_2 } \right)\bigcap
A } \right) \ ,
\end{equation*}
so any $r$ in the collection (\ref{BallMeasure}) can be used in
(\ref{SimpleSpecial}).

In the next lemma we show that the average defined above has the
desired properties,
\begin{lemma}\label{SimpDiffMeasurePropsSpecial}
The measure average of simply different sets in Definition
\ref{SimpleMeasureAverageNotSat} is well defined and satisfies
Properties \ref{PropClousure}-\ref{PropIncusion} in List of
Properties \ref{PropertiesList}.
\end{lemma}
\begin{proof}
First we consider the function $\psi\left( r \right) = \mu \left(
{O\left( {B,r} \right) \bigcap A} \right)$. It is easy to observe
that, $\psi\left(0\right)$ = $\mu(B)$ and that for $r$ large enough
$\psi\left(r\right)$ = $\mu(A)$. To show that $\psi$ is continuous,
we use a result from \cite{erdos1945some}, guaranteeing that for any
$B \subset \mathbb{R}^n$ and $\lambda > 0$, $\mu \left( {p:d\left(
{p,B} \right) = \lambda } \right) = 0$, followed by arguments as in
the proof of Lemma \ref{FunctionContinuityLemma}. By the continuity
of $\psi$, for any $t \in [0,1]$ the collection (\ref{BallMeasure})
is non-empty. Moreover the set $O\left( {B,r_{A,B} (t)}
\right)\bigcap A$ is Jordan measurable, consequently in view of
(\ref{jordanCIA}),
\begin{equation}\label{SimpDiffSpecialMeasureEq}
\mu \left( {tA \bigoplus \left( {1 - t} \right)B} \right) = t\mu
\left( A \right) + \left( {1 - t} \right)\mu \left( B \right) \ ,
\end{equation}
implying the measure property.

Next we observe that the closure, the interpolation and the
inclusion properties follow directly from Definition
\ref{SimpleMeasureAverageNotSat}.

Furthermore the construction in Definition
\ref{SimpleMeasureAverageNotSat} yields,
\begin{equation}\label{SimpDiffInclusionEq1}
sA \bigoplus \left( {1 - s} \right)B \subseteq tA \bigoplus \left(
{1 - t} \right)B \ , \ \ 0 \leq s \leq t \leq 1 \ .
\end{equation}
In all other cases of  $s \leq t$, $s,t \in \mathbb{R}$,
(\ref{SimpDiffInclusionEq1}) follows from
(\ref{SimpDiffInclusionEq}) and the inclusion property of the
measure average. In view of relations (\ref{IncludedDistance}) and
(\ref{SimpDiffInclusionEq1}), the metric property follows from the
measure property.

It remains to prove the submetric property. Let $h(x)$, $m_{A,B}$
and $M_{A,B}$ be defined as in the proof of Lemma
\ref{SimpDiffMeasureAverageProperties}.  From Corollary
\ref{FunctionContinuity}, we know that $h$ is continuous anywhere
except at $x = \frac{1}{2}$. From Lemma
\ref{AuxiliaryAverageProperties}, it follows that $h(0) = \mu(B)$,
$h(1) = \mu(A)$, so since $h$ is monotone non-decreasing, $h\left(
\frac{1}{2} \right) \in \left[\mu(B), \mu(A) \right]$. We conclude
that $z$ satisfying,
\begin{equation}\label{zInRange}
m_{A,B}< z <\mu(B) \text{ or } \mu(A) < z < M_{A,B} \ ,
\end{equation}
is in the range of $h(x)$.  This combined with
(\ref{SimpDiffSpecialMeasureEq}) gives a set of relations similar to
(\ref{SimpleMeasureTogether}), which in view of
(\ref{SimpDiffInclusionEq1}) and (\ref{IncludedDistance}) leads to
the submetric property. \qed
\end{proof}

Although the List of Properties \ref{PropertiesList} is satisfied by
the measure average of two simply different sets satisfying or not
satisfying the zero-measure condition, the distinction between the
two cases is important from the geometric point of view. This is so,
since the average defined in Definition
\ref{SimpleMeasureAverageSat} takes into account the geometric
structure of both sets, while the average in Definition
\ref{SimpleMeasureAverageNotSat} is biased towards the smaller set.

It follows from relations (\ref{SimpDiffInclusionEq}),
(\ref{SimpDiffInclusionEq1}), the interpolation property and the
inclusion property, that
\begin{equation}\label{NegativeT}
 tA \bigoplus (1-t)B \mathop  \subseteq  B, \ t < 0 \ ,
\end{equation}
\begin{equation} \label{IntervalT}
B \subseteq tA \bigoplus (1-t)B  \subseteq A, \ t \in [0,1] \ ,
\end{equation}
and
\begin{equation}\label{BiggerT}
B \subseteq A \subseteq tA \bigoplus (1-t)B, \ t > 1 \ .
\end{equation}

Relations (\ref{NegativeT})-(\ref{BiggerT}), imply that the average
$tA \bigoplus (1-t)B$ "cuts" into the set $B$ for $t < 0$,
interpolates between the two sets for $t \in [0,1]$ and expands
beyond the set $A$ for $t > 1$. This geometric behavior brings the
ideas of interpolation and extrapolation into the context of sets.
An example of the measure average of two simply different sets in
$\mathfrak{J}_2$ for varying values of the averaging parameter $t$
is given in Figure \ref{fig:DifferentValues}.

\begin{figure}[t!!!]
\begin{center}
\begin{tabular} {c c c c c c c}
\footnotesize $t = 1+\frac{1}{8}$ & \footnotesize  $t = 1$ (set $A$)
& \footnotesize $t = \frac{3}{4} $ &  \footnotesize $t = \frac{1}{2}
$ & \footnotesize $t = \frac{1}{4} $ & \footnotesize $t = 0$ (set
$B$) & \footnotesize $t = -\frac{1}{8}$
 \tabularnewline
\includegraphics[scale=0.8]{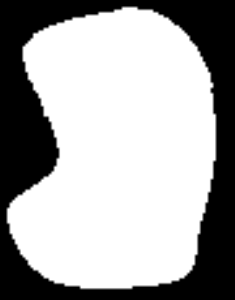} &
\includegraphics[scale=0.8]{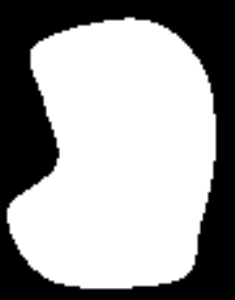} &
\includegraphics[scale=0.8]{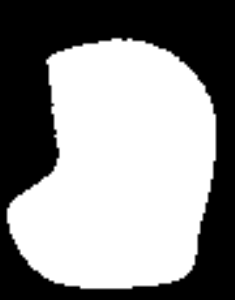} &
\includegraphics[scale=0.8]{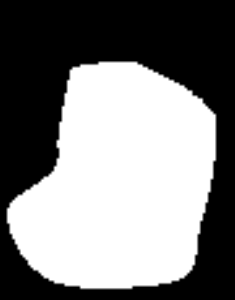} &
\includegraphics[scale=0.8]{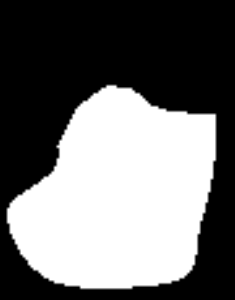}&
\includegraphics[scale=0.8]{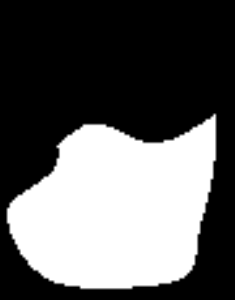} &
\includegraphics[scale=0.8]{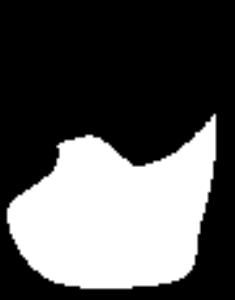}
\end{tabular}
\end{center}
\caption{The average $tA \bigoplus (1-t)B$ of simply different sets
for various values of the averaging parameter $t$.}
\label{fig:DifferentValues}
\end{figure}

\subsection{The measure average of "nested sets"}
We continue the construction with the case of $A,B \in
\mathfrak{J}_n$ such that $B \subset A$. We term two such sets
\emph{nested}. Formally, the measure average ${\bigoplus}$ for
simply different sets can be straightforwardly applied to any two
nested sets, preserving all the properties in List of Properties
\ref{PropertiesList}. However in such a construction, the function
$g(t)$ in (\ref{MeasureReparam}) dictating the change of the
averaging parameter in (\ref{SimpDiffMeasureAverage}) is global.
Consequently, when used for two nested sets $A,B$ that are not
simply different, the reparametrization creates interdependence
between different connected components of $A \setminus B$, and so
may lead to unsatisfactory results from the geometric point of view.
To reflect the local changes of the geometry of the two sets, we
decompose the average of $A,B$ into several averages of simply
different sets, using $B$ and each of the connected components of $A
\setminus B$. We denote the collection of all connected components
of a set $D$ by $\mathfrak{C}\left(D\right)$.

Let $A,B $ be nested sets, and assume at first that the number of
elements in $\mathfrak{C}\left(A \setminus B\right)$ is finite. For
any $C \in \mathfrak{C}\left(A \setminus B\right)$, we define using
the measure average $\bigoplus$ for simply different sets the set,
\begin{equation}\label{ComponentAverage}
R_{C,t} = t\left(B \bigcup C\right) \bigoplus \left( {1 - t}
\right)B \ .
\end{equation}
Note that $B \bigcup C$ and $ B$ are simply different sets.

The results of the averages of simply different sets obtained in
(\ref{ComponentAverage}) for all $C \in \mathfrak{C}\left(A
\setminus B\right)$ are merged into the average of $A$ and $B$,
taking into account the interpolation and the extrapolation induced
by relations (\ref{NegativeT})-(\ref{BiggerT}). For $t > 0$, it is
logical to take the union of the averages in
(\ref{ComponentAverage}), while for $t < 0$, it is logical to remove
from  $B$ the  union of the parts that are "cut" from $B$ by each
connected component (see (\ref{NegativeT})), which is equivalent to
the intersection of the averages in (\ref{ComponentAverage}). We
formalize the above procedure in the next definition,
\begin{definition}\label{ContainedMeasureAverage}
Let $A,B$ be nested sets such that the number of  elements in
$\mathfrak{C}\left(A \setminus B\right)$ is finite. The measure
average of $A,B$ is,
\begin{equation}\label{ContainedAverageEqu}
tA \bigoplus \left( {1 - t} \right)B = \left\{
{\begin{array}{*{20}c}
   {\begin{array}{*{20}c}
   {\bigcup\limits_{C \in \mathfrak{C}\left( {A \setminus B} \right)} {R_{C,t} } } & \qquad \ \ \ \ \ \ {t \geqslant 0}  \\

 \end{array} }  \\
   {\begin{array}{*{20}c}
   \operatorname{ci}\left( {\bigcap\limits_{C \in \mathfrak{C}\left( {A \setminus B} \right)} {R_{C,t} } }\right) & \qquad {t < 0} \\

 \end{array} }  \\

 \end{array} } \right.
\end{equation}
\end{definition}
The operator $\operatorname{ci}$ is applied in case $t < 0$, because
regular compact sets are not closed under finite intersections.
However since ${\bigcap\limits_{C \in \mathfrak{C}\left( {A
\setminus B} \right)} {R_{C,t} } }$ is Jordan measurable,
from(\ref{jordanCIA})
\begin{equation*}
\mu \left( {\bigcap\limits_{C \in \mathfrak{C}\left( {A \setminus B}
\right)} {R_{C,t} } } \right) = \mu \left( {\operatorname{ci} \left(
{\bigcap\limits_{C \in \mathfrak{C}\left( {A \setminus B} \right)}
{R_{C,t} } } \right)} \right)  \ .
\end{equation*}

In the study of the properties of the measure average in Definition
\ref{ContainedMeasureAverage}, we need the simple observation below,
\begin{lemma}\label{SetsIncusion}Let $F_1,...,F_n$, $E_1,...,E_n$ be sets such that $F_i \subseteq E_i, i = 1,...,n$, then
\begin{enumerate}
\item \label{Inclusion1}$\left( {\bigcup\limits_{i = 1}^n {E_i } } \right)\setminus \left( {\bigcup\limits_{i = 1}^n {F_i } } \right) \subseteq \bigcup\limits_{i = 1}^n {\left( {E_i \setminus F_i } \right)}$

\item \label{Inclusion2} $\left( {\bigcap\limits_{i = 1}^n {E_i } } \right)\setminus \left( {\bigcap\limits_{i = 1}^n {F_i } } \right) \subseteq \bigcup\limits_{i = 1}^n {\left( {E_i \setminus F_i } \right)} $\\\\

Under the stronger assumption that for all $i,j \in \left\{
{1,...,n} \right\},F_j  \subseteq E_i$,
\item \label{Inclusion3} $\left( {\bigcup\limits_{i = 1}^n {E_i } } \right)\setminus \left( {\bigcap\limits_{i = 1}^n {F_i } } \right) \subseteq \bigcup\limits_{i = 1}^n {\left( {E_i \setminus F_i } \right)} $

\end{enumerate}
\end{lemma}
\begin{proof} Properties \ref{Inclusion1}-\ref{Inclusion2} are immediate to verify. To observe  Property \ref{Inclusion3},
let $p \in \left( {\bigcup\limits_{i = 1}^n {E_i } }
\right)\setminus \left( {\bigcap\limits_{i = 1}^n {F_i } } \right)$,
then $p \in E_l$ for some $l$. Now if $p \notin F_l$, then $p \in
E_l \setminus F_l$. If $p \in F_l$, then due to the stronger
assumption, $p \in E_i$ for all $i$. Since there is $j$ such that $p
\notin F_j$, $p \in E_j \setminus F_j$. \qed We can now show that,
\begin{lemma}\label{ContainedMeasureAverageProps}
The measure average of nested sets satisfies Properties
\ref{PropClousure}-\ref{PropIncusion} in List of Properties
\ref{PropertiesList}.
\end{lemma}
\proof The closure property, the interpolation property and the
inclusion property follow immediately from Definition
\ref{ContainedMeasureAverage} and the corresponding properties of
the measure average of simply different sets. Assume without loss of
generality that $s \leq t$, then by relations
(\ref{SimpDiffInclusionEq}) and (\ref{SimpDiffInclusionEq1}),
$R_{C,s}  \subseteq R_{C,t}$, leading to
\begin{equation}\label{NestedInclusionEq}
sA \bigoplus \left( {1 - s} \right)B \subseteq tA \bigoplus \left(
{1 - t} \right)B \ .
\end{equation}
To prove the submetric property, we have to consider three cases:
(i). $0 \leq s \leq t$ (ii).$s \leq t \leq 0$  and (iii).$s < 0 \leq
t$. Observe that the averages $sA \bigoplus \left( {1 - s}
\right)B$, $tA \bigoplus \left( {1 - t} \right)B$ in cases
(i)-(iii), when written in terms of $\left\{ {R_{C,s} :C \in
\mathfrak{C}\left( {A\backslash B} \right)} \right\}$ and $\left\{
{R_{C,t} : C\in \mathfrak{C}\left( {A\backslash B} \right)}
\right\}$, correspond to cases 1 - 3 in Lemma \ref{SetsIncusion}. In
particular, the assumptions required for case \ref{Inclusion3} in
Lemma \ref{SetsIncusion} are satisfied due to relations
(\ref{NegativeT})-(\ref{BiggerT}). We conclude that in view of
(\ref{NestedInclusionEq}) and Lemma \ref{SetsIncusion},
\begin{equation*}
d_\mu  \left( {tA \bigoplus \left( {1 - t} \right)B, sA \bigoplus
\left( {1 - s} \right)B} \right)= \mu \left( {tA \bigoplus \left( {1
- t} \right)B\setminus s A \bigoplus \left( {1 - s} \right)B}
\right) \leqslant \mu \left( {\bigcup\limits_{C \in
\mathfrak{C}\left( {A \setminus B} \right)} {R_{C,t} }  \setminus
R_{C,s} } \right) \ .
\end{equation*}
Now by the submetric property of the measure average of simply
different sets,
\begin{equation*}
\sum\limits_{C \in \mathfrak{C}\left( {A\setminus B} \right)} {\mu
\left( {R_{C,t} \setminus R_{C,s} } \right)}  \leq \sum\limits_{C
\in \mathfrak{C}\left( {A\setminus B} \right)} {\left| {t - s}
\right|} \mu \left( C \right) = \left| {t - s} \right|\left( {\mu
\left( A \right) - \mu \left( B \right)} \right)\ ,
\end{equation*}
yielding the submetric property.

To prove the measure property we take $t \in [0,1]$, and observe
that for $C_1,C_2 \in \mathfrak{C}\left(A\setminus B\right)$, $C_1
\neq C_2$ we have by (\ref{IntervalT}) that $R_{C_1 ,t}  \bigcap
R_{C_2 ,t}  = B$. Therefore,
\begin{equation*}
\mu \left( {tA \bigoplus \left( {1 - t} \right)B} \right) = \mu
\left( {\bigcup\limits_{C \in \mathfrak{C}\left( {A \setminus B}
\right)} {R_{C,t} } } \right) = \mu \left( B \right) +
\sum\limits_{C \in \mathfrak{C}\left( {A \setminus B} \right)} {\mu
\left( {R_{C,t}  \setminus B} \right)} \ ,
\end{equation*}
and by the measure property of the measure average of simply
different sets,
\begin{equation*}
\mu \left( {tA \bigoplus \left( {1 - t} \right)B} \right) = \mu
\left( B \right) + \sum\limits_{C \in \mathfrak{C}\left( {A
\setminus B} \right)} {t\mu \left( C \right)}  = \mu \left( B
\right) + t\left( {\mu \left( A \right) - \mu \left( B \right)}
\right) \ .
\end{equation*}
Finally the metric property follows from the measure property, by
(\ref{IncludedDistance}) and (\ref{NestedInclusionEq}). \qed
\end{proof}
The measure average of nested sets satisfies  relations
(\ref{NegativeT})-(\ref{BiggerT}) as well.

Although the measure average of $A$ with itself is not defined by
Definition \ref{ContainedMeasureAverage}, it follows from continuity
arguments.
\begin{remark}
If follows from the interpolation property and the metric property,
that the sequence of measure averages $tA_i \bigoplus \left( {1 - t}
\right)B$, with $A_i = B \bigcup D_i$ and $\mathop {\lim }\limits_{i
\to \infty } \mu \left( {D_i } \right) = 0$, satisfies $\mathop
{\lim }\limits_{i \to \infty } tA_i \bigoplus \left( {1 - t}
\right)B = B$. Therefore by continuity, we define for any $t \in
\mathbb{R}$, $A \in \mathfrak{J}_n$, $tA \bigoplus \left( {1 - t}
\right)A = A$.
\end{remark}

For the sake of completeness, we consider the case of nested $A,B$,
when $A \setminus B$ has an infinite number of connected components.
In this case, since the measure of $A$ is bounded, there is only a
finite number of connected components of $A$ with measure greater
that a preassigned $\varepsilon > 0$, and all connected components
of $A \setminus B$ with measure smaller than $\varepsilon$ are
joined into one set,
\begin{equation*}
U_{A,B,\varepsilon }  = \bigcup\limits_{C \in \mathfrak{C}\left(
{A,B} \right),\mu \left( C \right) < \varepsilon } C \ .
\end{equation*}
The set $U_{A,B,\varepsilon }$ is treated as a "single component" in
Definition \ref{ContainedMeasureAverage}. It is not difficult to
show that $B \bigcup U_{A,B,\varepsilon }$ is in $\mathfrak{J}_n$.
One can verify that all properties in List of Properties
\ref{PropertiesList} are preserved in this case.

\subsection{The measure average of general sets}
We are now in position to define the measure average $tA \bigoplus
\left( {1 - t} \right)B$ of two general sets $A,B \in
\mathfrak{J}_n$. The average is decomposed into two averages of
nested sets,
\begin{equation}\label{R1}
R_{1,t}  = tA \bigoplus (1 - t)\operatorname{ci}\left( {A\bigcap B }
\right) \ ,
\end{equation}
and
\begin{equation}\label{R2}
R_{2,t}  = (1 - t)B \bigoplus t\operatorname{ci}\left( {A\bigcap B }
\right) \ .
\end{equation}
The two averages are merged preserving the geometry of the
interpolation and the extrapolation of sets.
\begin{definition}\label{GeneralAverage}
Let $A,B \in \mathfrak{J}_n$, the measure average of $A,B$ with $t
\in \mathbb{R}$ is,
\begin{equation}
tA \bigoplus (1 - t)B = \left\{ {\begin{array}{l}
   {R_{1,t} \bigcup {R_{2,t} } }  \medskip\\
   {{\operatorname{ci}\left( {R_{1,t} \backslash A \bigcap B} \right) \bigcup R_{2,t}} }  \medskip\\
   {(1 - t)A \bigoplus tB}  \\
\end{array}
\qquad \qquad
\begin{array}{l}
   {t \in \left[ {0,1} \right]}  \medskip\\
   {t > 1}  \medskip\\
   {t < 0}  \\
\end{array}} \right.
\end{equation}
with the sets $R_{1,t}$,$R_{2,t}$ defined in (\ref{R1})-(\ref{R2}).
\end{definition}

Notice that for $t > 1$, $A \bigcap B$ is removed from $R_{1,t}$, so
that the "cutting" from $A \bigcap B$ by the extrapolation in
$R_{2,t}$ will affect the resulting average.

\begin{remark}\label{Symmetry}
It follows from Definition \ref{GeneralAverage} that for any $t \in
\mathbb{R}$, $A,B \in \mathfrak{J}_n$, $tA \bigoplus (1 - t)B =
\left( {1 - t} \right)B \bigoplus tA$.
\end{remark}

\begin{theorem}\label{GeneralMeasureAverageProperties}
The measure average of any two sets $A,B \in \mathfrak{J}_n$
satisfies Properties \ref{PropClousure}-\ref{PropIncusion} in List
of Properties \ref{PropertiesList}.
\end{theorem}
\begin{proof} The closure, the interpolation and the inclusion properties follow from the similar properties
of the measure average of nested sets. Next, we observe that from
the inclusion property in the nested case, we have for $t \in
[0,1]$,
\begin{equation}\label{OnlyIntersection}
R_{1,t} \bigcap R_{2,t}  = \operatorname{ci}\left(A \bigcap B\right)
\ ,
\end{equation}
so,
\begin{equation}\label{MeasureMath}
\mu \left( {tA \bigoplus \left( {1 - t} \right)B} \right) = \mu
\left( {R_{1,t} } \right) + \mu \left( {R_{2,t} } \right) - \mu
\left( {A\bigcap B } \right) \ .
\end{equation}
Using (\ref{MeasureMath}) and the measure property in the nested
case we obtain that for $t \in [0,1]$,
\begin{equation*}
\mu \left( {tA \bigoplus \left( {1 - t} \right)B} \right) = t\mu
\left( A \right) + \left( {1 - t} \right)\mu \left( {A\bigcap B }
\right) + \left( {1 - t} \right)\mu \left( B \right) + t\mu \left(
{A\bigcap B } \right) - \mu \left( {A\bigcap B } \right)\ ,
\end{equation*}
which yields the measure property. The metric property is proved
using (\ref{OnlyIntersection}) and the metric property of the
measure average in the nested case.

To prove the submetric property, we observe that for $s \leq t$,
$R_{1,s}  \subseteq R_{1,t} $ and $R_{2,t}  \subseteq R_{2,s} $.
Consequently,
\begin{equation}
d_\mu  \left( {sA \bigoplus \left( {1 - s} \right)B,tA \bigoplus
\left( {1 - t} \right)B} \right) \leq d_\mu  \left( {R_{1,s}
,R_{1,t} } \right) + d_\mu  \left( {R_{2,s} ,R_{2,t} } \right) \ ,
\end{equation}
leading to the submetric property in view of the submetric property
of the measure average in the nested case.
\end{proof}
\qed

As an immediate consequence of the metric property of the measure
average, we obtain a result about the approximation of H{\"o}lder
continuous SVFs by the piecewise interpolant based on the measure
average. We term a SVF $F$ H{\"o}lder-$\nu$ continuous if,
\begin{equation}
d_\mu  \left( {F\left( {t_1 } \right),F\left( {t_2 } \right)}
\right) \leq C\left| {t_1  - t_2 } \right|^\nu \ ,
\end{equation}
where $C$ is a constant (termed the H{\"o}lder constant of $F$)
depending on $F$ but not on $t_1,t_2$ and $\nu \in
\left(0,1\right]$. A H{\"o}lder-$1$ continuous SVF is also termed
Lipschitz continuous.
\begin{corollary}\label{PieceApprox} Let $F:\left[ {0,1} \right] \to \mathfrak{J}_n $ be H{\"o}lder-$\nu$ continuous. We define $P_N\left( x \right):[0,1] \to \mathfrak{J}_n $ to be
the piecewise interpolant,
\begin{equation}\label{PieceApproxEq}
P_N\left( x \right) = \left( {\frac{x} {h} - i} \right)F\left( {ih}
\right) \bigoplus \left( {\left( {i + 1} \right) - \frac{x} {h}}
\right)F\left( {\left( {i + 1} \right)h} \right),\ \ x \in \left[
{ih,\left( {i + 1} \right)h} \right], \ h = 1/N, \ i = 0,...N + 1 \
.
\end{equation}
Then for any $x \in [0,1]$,
\begin{equation*}
d_\mu \left( {F\left( x \right),P_N\left( x \right)} \right) \leq
Ch^\nu  \ ,
\end{equation*}
where $C$ is the H{\"o}lder constant of $F$.
\end{corollary}

In the following sections we study subdivision methods based on the
measure average of sets.

\section{Spline subdivision schemes adapted to sets with the measure average}\label{sectionSplineSubdivision}
In this section, we use the measure average in the adaptation to
sets of spline subdivision schemes (\emph{see.e.g}
\cite{dyn2003subdivision}, Section 3.1) . First, the refinement rule
is expressed by repeated binary averages using the Lane-Riesenfeld
algorithm \cite{lane1980theoretical}. Binary averages of numbers are
then replaced by the measure averages of sets. Our approach is
similar to \cite{dyn2002spline}, where spline subdivision schemes
are adapted to sets using the metric average.

An $m$-degree spline subdivision scheme ($ m \geq 1$) in the
real-valued setting refines the numbers,
\begin{equation*}
\left\{ {f_i^k :i \in \mathbb{Z}} \right\} \subset \mathbb{R} ,
\end{equation*}
according to the refinement rule,
\begin{equation}\label{splineRule}
f_i^{k + 1}  = \sum\limits_{j \in \mathbb{Z}} {a_{\left\lceil
{\frac{{m + 1}} {2}} \right\rceil  + i - 2j}^{\left( m \right)}
f_j^k } \ \ ,i \in \mathbb{Z}\ \ ,k = 0,1,2,...,
\end{equation}
where $\lceil \cdot \rceil$ is the \emph{ceiling function}, $a_{l}^{(m)}  = \left( {\begin{array}{*{20}c}   {m + 1}  \\   l  \\
 \end{array} } \right)/2^m$ for $l = 0,1,...,m + 1$,  and $a_{l}^{(m)}  = 0$ for $l \in \mathbb{Z}\setminus \left\{ {0,1,...,m + 1} \right\}$.

The case $m = 2$ is the Chaikin subdivision scheme with the
refinement rule,
\begin{equation*}
f_{2i}^{k + 1}  = \frac{3}{4}f_i^k  + \frac{1}{4}f_{i + 1}^k \ ,
\end{equation*}
\begin{equation*}
f_{2i + 1}^{k + 1}  = \frac{1}{4}f_i^k  + \frac{3}{4}f_{i + 1}^k \ .
\end{equation*}
Chaikin scheme is the simplest which generates $C^1$ limits, and is
widely used.

At each level $k$, the piecewise linear interpolant to the data
$\left( {2^{ - k} i,f_i^k } \right),k \in \mathbb{Z}$ is defined on
$\mathbb{R}$ by,
\begin{equation}\label{realfk}
f_k \left( x \right) = \lambda \left( x \right)f_i^k  + \left( {1 -
\lambda \left( x \right)} \right)f_{i + 1}^k \;,\qquad i2^{ - k}
\leq x \leq \left( {i + 1} \right)2^{ - k}  \ ,
\end{equation}
where $\lambda \left( x \right) = \left( {i + 1} \right) - x2^k$.
The sequence $\left\{ {f_k }\left(x\right) \right\}$ converges
uniformly to a continuous function $f^\infty\left(x\right)$, which
is the spline of degree $m$ with the control points $\left( {i,f_i^0
} \right),i \in \mathbb{Z}$.

The Lane-Riesenfeld algorithm evaluates (\ref{splineRule}), by first
doubling the values at level $k$ and then performing $m$ steps of
repeated averages. We apply the above procedure to sets, by
replacing averages of numbers with the measure average of sets. We
combine the doubling step with one averaging step. So first the
sequence of sets at level $k$, $\left\{ {F_i^k :i \in \mathbb{Z}}
\right\}$ is refined using the measure average,
\begin{equation}\label{SetSplineRule1}
F_{2i}^{k + 1,0}  = F_i^k, \ \ F_{2i + 1}^{k + 1,0}  =
\frac{1}{2}F_i^k  \bigoplus \frac{1}{2}F_{i + 1}^k,\ \ i \in
\mathbb{Z} \ .
\end{equation}
Then for $1 \leq  j \leq m - 1$, the sequence $\left\{ {F_i^{k + 1,j
- 1} :i \in \mathbb{Z}} \right\}$ is replaced by the measure
averages of pairs of consecutive sets,
\begin{equation}\label{SetSplineRule2}
F_i^{k + 1,j}  = \frac{1} {2}F_i^{k + 1,j - 1}  \bigoplus \frac{1}
{2}F_{i + 1}^{k+1,j - 1} ,\qquad i \in \mathbb{Z}\ .
\end{equation}
Finally, the refined sets at level $k+1$ are given by,
\begin{equation}\label{SetSplineRule3}
\begin{array}{*{20}c}
 F_i^{k + 1}  = F_{i - \left\lfloor {\frac{{m - 1}}
{2}} \right\rfloor }^{k + 1,m - 1}  \ , \qquad i \in \mathbb{Z} \ .
 \end{array}
\end{equation}
For the analysis of the above family of subdivision schemes we use,
similarly to the real-valued case, the piecewise interpolant $F_k
\left( x \right)$ in terms of the measure average through the sets
al the $k$-th level,
\begin{equation}\label{Fk}
F_k \left( x \right) = \lambda \left( x \right)F_i^k  \bigoplus
\left( {1 - \lambda \left( x \right)} \right)F_{i + 1}^k \ , \qquad
i2^{ - k} \leq x \leq  \left( {i + 1} \right)2^{ - k} \ ,
\end{equation}
 where
$\lambda \left( x \right) = \left( {i + 1} \right) - x2^k $.

The following results on convergence and approximation order of
spline subdivision schemes based on the measure average are
analogous to results on spline subdivision schemes bases on the
metric average in \cite{dyn2001spline}.
\begin{theorem} \label{SplineConvergence} The sequence of set-valued functions $\left\{ {F_k \left( x \right)} \right\}_{k \in \mathbb{Z} + }$ converges uniformly to a continuous set-valued function $F^\infty\left( x \right) : \mathbb{R} \to \mathfrak{L}_n$,
which is Lipschitz continuous relative to the symmetric difference
metric with the Lipschitz constant $L  = \mathop {\sup }\limits_i
d_\mu  \left( {F_i^0 ,F_{i + 1}^0 } \right)$.
\end{theorem}
\begin{theorem}\label{SplineKApproximation} Let $G : \mathbb{R} \to \mathfrak{L}_n$ be Lipschitz continuous,
and let the initial sets be given by $F_i^0  = G\left( {\delta  +
ih} \right) \in \mathfrak{J}_n,$  $i \in \mathbb{Z}$ with $\delta
\in [0,h)$ and $h > 0$. Then,
\begin{equation}
d_\mu \left( {G \left( x \right),F_k \left( x \right)} \right) \leq
Ch \ ,
\end{equation}
where $F_k\left(x\right)$ is given by (\ref{Fk}) and $C$ is a
constant depending on the degree of the scheme.
\end{theorem}
\begin{corollary}\label{SplineInfApproximation}Under the assumptions of Theorem
\ref{SplineKApproximation}, the distance between the original
set-valued function $G\left(x\right)$ and the limit set-valued
function $F^\infty \left( x \right)$  is bounded by,
\begin{equation*}
\mathop {\max }\limits_x d_\mu  \left( {F^\infty  \left( x
\right),G\left( x \right)} \right) \leq Ch \ .
\end{equation*}
\end{corollary}
The properties of the metric average relative to the Hausdorff
metric, used in the proofs of results analogues to the above
results, are also possessed by the measure average relative to the
symmetric difference metric. Therefore the proofs of Theorems
\ref{SplineConvergence},\ref{SplineKApproximation} and Corollary
\ref{SplineInfApproximation} are similar to the proofs of Theorems
4.3, 4.4 and Corollary 4.5 in \cite{dyn2001spline} respectively, and
are omitted here. Theorem \ref{SplineKApproximation} and  Corollary
\ref{SplineInfApproximation} can be straightforwardly extended to
H\"older-$\nu$ SVFs to obtain approximation order
$O\left(h^\nu\right)$.

Next we use the measure property and the inclusion property, which
are specific to the measure average, to derive further properties of
spline subdivision schemes based on the measure average. As a
consequence of the measure property we have,
\begin{corollary}\label{GeneralMeasureTransition}Let $S$ be an averaging rule defined
for a sequence $\left\{ {f_i } \right\}_{i \in \mathbb{Z}}  \subset
R$ by,
\begin{equation*}
S\left( {\left\{ {f_i } \right\}_{i \in \mathbb{Z}} } \right) =
\sum\limits_i {a_i } f_i \ ,
\end{equation*}
with $a_i  \geq 0,\sum\limits_i {a_i }  = 1$. Let $S^*$ be an
adaptation of $S$ to sets by representing $S$ as a sequence of
repeated binary averages of numbers with non-negative averaging
parameters, and replacing averages of numbers by the measure
averages of sets. Then for any sequence of sets
$\left\{F_i^0\right\}_{i \in \mathbb{Z}}  \subset \mathfrak{J}_n$,
\begin{equation*}
\mu \left( {S^* \left( {\left\{ {F_i } \right\}_{i \in \mathbb{Z}} }
\right)} \right) = S\left( {\left\{ {\mu \left( {F_i } \right)}
\right\}_{i \in \mathbb{Z}} } \right) \ .
\end{equation*}
\end{corollary}
Note that the result of Corollary \ref{GeneralMeasureTransition} is
independent of the specific representation of the averaging rule $S$
by repeated binary averages.

It follows from Theorem \ref{SplineConvergence} and the definition
of the metric $d_\mu\left(\cdot, \cdot\right)$ that,
\begin{equation}\label{MeasureLimit}
\mu \left( {F^\infty  \left( x \right)} \right) = \mathop {\lim
}\limits_{k \to \infty } \mu \left( {F_k \left( x \right)} \right)\
.
\end{equation}
Corollary \ref{GeneralMeasureTransition} together with
(\ref{MeasureLimit}) leads to,
\begin{corollary}\label{MeasureTransition}. Let $\left\{F_i^0\right\}_{i \in \mathbb{Z}}  \subset \mathfrak{J}_n$. Let
$F^\infty$ be the limit SVF of the set-valued spline subdivision
scheme applied to $\left\{F_i^0\right\}_{i \in \mathbb{Z}}$ and let
$f^\infty$  be the limit function of the real-valued spline
subdivision scheme applied to $\left\{ \mu\left(F_i^0\right)
\right\}_{i \in \mathbb{Z}}$ . Then,
\begin{equation*}
\mu \left( {F^\infty  \left( x \right)} \right) = f^\infty  \left( x
\right) \ .
\end{equation*}
\end{corollary}

Next we state several results concerning set-valued spline
subdivision schemes applied to monotone data. It is well knows that
real-valued spline subdivision schemes are monotonicity preserving.
Due to the inclusion property of the measure average and relations
(\ref{SetSplineRule1})-(\ref{SetSplineRule3}), spline subdivision
schemes adapted to sets with the measure average are also
monotonicity preserving. A sequence of sets ${\left\{ {F_i }
\right\}_{i \in \mathbb{Z}} }$ is termed monotone non-decreasing
(non-increasing) if $F_i  \subseteq F_{i + 1} \left( {F_i  \supseteq
F_{i + 1} } \right)$. With a similar definition for a non-decreasing
(non-increasing) set valued function we obtain,

\begin{corollary}\label{Monotonicity}Let $\left\{F_i^0\right\}_{i \in \mathbb{Z}}  \subset \mathfrak{J}_n,$ be monotone
non-decreasing (non-increasing). Let $F^\infty$ be the limit SVF of
the set-valued spline subdivision scheme applied to
$\left\{F_i^0\right\}_{i \in \mathbb{Z}}$. Then $F^{\infty}$ is
monotone non-decreasing (non-increasing).
\end{corollary}

The notion of the speed of a curve in a metric space (\emph{see
e.g.} \cite{burago2001course}, Chapter 2), can be used as an
indication of the "smoothness" of a set-valued function. For a
real-valued $f$ the speed at a point $x$ is,
\begin{equation*}
v_f \left( x \right) = \mathop {\lim }\limits_{\varepsilon \to 0 }
\frac{{\left| {f\left( x \right) - f\left( {x + \varepsilon }
\right)} \right|}} {{\left| \varepsilon  \right|}} \ ,
\end{equation*}
whenever the limit exists. For differentiable $f$, $v_f$ is the
absolute value of the derivative of $f$. We define the velocity of a
SVF $F$,
\begin{equation*}
v_F \left( x \right) = \mathop {\lim }\limits_{\varepsilon \to 0 }
\frac{{d_\mu  \left( {F\left( x \right),F\left( {x + \varepsilon }
\right)} \right)}} {{\left| \varepsilon  \right|}} \ .
\end{equation*}
By combining relation (\ref{IncludedDistance}) with Corollaries
\ref{MeasureTransition} and \ref{Monotonicity} we arrive at,
\begin{corollary}\label{SplineMetricDerivative} Let $\left\{F_i^0\right\}_{i \in \mathbb{Z}}  \subset \mathfrak{J}_n,$ be monotone
non-decreasing (non-increasing). Let $F^\infty$ and $f^\infty$ be
defined as in Corollary \ref{MeasureTransition}, then
\begin{equation*}
v_{F^\infty}\left(x\right)  = v_{f^\infty}\left(x\right)  \ .
\end{equation*}
\end{corollary}
Under the assumptions of Corollary \ref{SplineMetricDerivative}, we
have that for the spline subdivision of degree $m \geq 2$,
$v_{F^\infty}$ is continuous and has continuous derivatives up to
order $m-2$.

In the next section we adapt to SVFs the 4-point interpolatory
subdivision scheme using the measure average with both positive and
negative averaging parameters.

\section{The 4-point subdivision scheme adapted to sets with the measure average}\label{section4pointScheme}
In the real-valued setting, the 4-point subdivision scheme is
defined by the following refinement rule
\begin{equation}\label{InterpolationRule}
f_{2i}^{k + 1}  = f_i^k  \ ,
\end{equation}
and
\begin{equation}\label{FourPoint}
f_{2i+1}^{k + 1}  =  - w\left( {f_{i - 1}^k  + f_{i + 2}^k }
\right)\\
 + \left( {1/2 + w} \right)\left( {f_i^k  + f_{i + 1}^k }
\right) \ ,
\end{equation}
repeatedly applied to  refine the values $\left\{ {f_i^k :i \in
\mathbb{Z}} \right\} \subset \mathbb{R}$ for $k = 0,1,2...$. Here
$w$ is a \emph{fixed tension parameter}. Usually $w$ is chosen to be
$1/16$, since this value yields the highest approximation order
\cite{dyn19874} and the maximal Holder exponent of the first
derivative of the limit function \cite{deslauriers1989symmetric}.
The rule can also be applied to refine a sequence of points in
$\mathbb{R}^n$.

One can see that the coefficients in (\ref{FourPoint}) sum to one,
so it is a weighted average of the four values
$f_{i-1},...,f_{i+2}$. Rewriting the insertion rule
(\ref{FourPoint}) in terms of binary weighted averages as,
\begin{eqnarray}\label{FourPointRewritten1}
  f_{2i + 1}^{k + 1}  & = & \frac{1}{2}\left( {\left( { - 2w} \right)f_{i - 1}^k  + \left( {1 + 2w} \right)f_i^k } \right) \\
                    \nonumber  & + & \frac{1} {2}\left( {\left( { - 2w} \right)f_{i + 2}^k  + \left( {1 + 2w}\right)f_{i + 1}^k } \right)
                    \ ,
\end{eqnarray}
we adapt the refinement rule defined by (\ref{InterpolationRule})
and (\ref{FourPointRewritten1}) to sets by replacing binary averages
of points by the measure averages of sets.

The refinement rule for the sets $\left\{ {F_i^k :i \in } \right\}
\subset \mathfrak{J}_n $ is,
\begin{equation}\label{RefinementRuleForSets1}
F_{2i}^{k + 1}  = F_i^k  \ ,
\end{equation}
and
\begin{equation}\label{RefinementRuleForSets2}
F_{2i + 1}^{k + 1}  = \frac{1} {2}E_{2i + 1}^{k + 1}  \bigoplus
\frac{1} {2}H_{2i + 1}^{k + 1} \ ,
\end{equation}
with
\begin{equation}E_{2i + 1}^{k + 1}  = \left( { - 2w } \right)F_{i - 1}^k
\bigoplus \left( {1 + 2w } \right)F_i^k \ ,
\end{equation}
 and
\begin{equation}H_{2i + 1}^{k + 1} = \left( { - 2w } \right)F_{i + 2}^k
\bigoplus \left( {1 + 2w } \right)F_{i + 1}^k \ .
\end{equation}
Note that the subdivision with the refinement rule
(\ref{RefinementRuleForSets1})-(\ref{RefinementRuleForSets2}) is
interpolatory.

First we study the convergence of the scheme.
\begin{lemma}\label{SetConvergence} Let $\left\{ {F_i^k :i \in \mathbb{Z}}
\right\} \subset \mathfrak{J}_n$  and define,
\begin{equation*}
 d_k  = \mathop
 {\sup
}\limits_{i \in \mathbb{Z}} d_\mu \left( {F_i^k ,F_{i + 1}^k }
\right) \ ,
\end{equation*}
then
\begin{equation}\label{GeometricDk}
d_k  \leq \left( {\frac{1} {2} + 4w} \right)^k d_0 \ .
\end{equation}
\end{lemma}
\begin{proof} By the interpolation property and
the submetric property of the measure average (see List of
Properties \ref{PropertiesList}),
\begin{equation}\label{fromEtoFi}
d_\mu \left( {E_{2i + 1}^{k + 1} ,F_i^k } \right) \leq 2wd_\mu
\left( {F_{i - 1}^k ,F_i^k } \right) \leq 2wd_k \ ,
\end{equation}
and,
\begin{equation}
d_\mu \left( {H_{2i + 1}^{k + 1} ,F_{i + 1}^k } \right) \leq 2wd_\mu
\left( {F_{i + 1}^k ,F_{i + 2}^k } \right) \leq 2wd_k \ .
\end{equation}
Therefore,
\begin{equation}
d_\mu \left( {E_{2i + 1}^{k + 1} ,H_{2i + 1}^{k + 1} } \right) \leq
d_\mu \left( {E_{2i + 1}^{k + 1} ,F_i^k } \right) + d_\mu \left(
{F_i^k ,F_{i + 1}^k } \right) + d_\mu \left( {H_{2i + 1}^{k + 1}
,F_{i + 1}^k } \right) \leq \left( {1 + 4w} \right)d_k \ ,
\end{equation}
from which we obtain by using (\ref{RefinementRuleForSets2}) and the
metric property,
\begin{equation}\label{fromEtoFipluls}
d_\mu \left( {F_{2i + 1}^{k + 1} ,E_{2i + 1}^{k + 1} } \right) \leq
\left( {\frac{1} {2} + 2w} \right)d_k \ .
\end{equation}
Thus by the triangle inequality we get from (\ref{fromEtoFi}) and
(\ref{fromEtoFipluls}),
\begin{equation*}
d_\mu \left( {F_{2i}^{k + 1} ,F_{2i + 1}^{k + 1} } \right) \leq
\left( {\frac{1} {2} + 4w} \right)d_k \ .
\end{equation*}
Similarly one gets,
\begin{equation*}
d_\mu \left( {F_{2i + 1}^{k + 1} ,F_{2i + 2}^{k + 1} } \right) \leq
\left( {\frac{1} {2} + 4w} \right)d_k \ .
\end{equation*}
Therefore
\begin{equation*}
 d_{k + 1}  \leq \left( {\frac{1} {2} + 4w}
\right)d_k \ ,
\end{equation*}
and (\ref{GeometricDk}) holds. \qed
\end{proof}
 We conclude from Lemma \ref{SetConvergence},
\begin{lemma}\label{LemmaDistanceAtT}Let $\left\{ {F_k \left( x \right)} \right\}_{k \in \mathbb{Z} +
}$ be the sequence of piecewise interpolants defined as in
(\ref{Fk}), then
\begin{equation}\label{DistanceAtT}
d_\mu \left( {F_k \left( x \right),F_{k + 1} \left( x \right)}
\right) \leq C\left( {\frac{1} {2} + 4w} \right)^{k } \ ,
\end{equation}
with $C = d_0\left( {1 + 4w} \right)$.
\end{lemma}
\begin{proof}
For $i2^{ - k} \leq x \leq  \left( {i + \frac{1}{2}} \right)2^{ -
k}$, using $F_i^k = F_{2i}^{k+1}$,
\begin{equation*}
d_\mu  \left( {F_k \left( x \right),F_{k + 1} \left( x \right)}
\right) \leq d_\mu  \left( {F_k \left( x \right),F_i^k } \right) +
d_\mu  \left( {F_{2i}^{k + 1} ,F_{k + 1} \left( x \right)} \right)
\leq \frac{1} {2}d_k  + d_{k + 1}  \ ,
\end{equation*}
and in view of (\ref{GeometricDk}),
\begin{equation*}
d_\mu \left( {F_k \left( x \right),F_{k+1} \left(x\right)} \right)
\leq \frac{1} {2}d_0 \left( {\frac{1} {2} + 4w} \right)^{k }  + d_0
\left( {\frac{1} {2} + 4w} \right)^{k+1} \ ,
\end{equation*}
leading to the claim of the lemma in this case. A similar argument
for $\left( {i + \frac{1} {2}} \right)2^{ - k}  \leq x \leq \left(
{i + 1} \right)2^{ - k}$ with $F_i^k$ replaced by $F_{i+1}^k$
completes the proof. \qed
\end{proof}

\begin{theorem}\label{FourPointConvergence} The sequence of set-valued functions $\left\{ {F_k \left( x \right)} \right\}_{k \in \mathbb{Z} + }
$ converges uniformly to a continuous set-valued function
$F^\infty\left( x \right) : \mathbb{R} \to \mathfrak{L}_n$ whenever
$w < \frac{1}{8}$.
\end{theorem}
\begin{proof} By definition the functions $\left\{ {F_k \left( x \right)} \right\}_{k \in \mathbb{Z} + } $
are continuous. By the triangle inequality,
\begin{equation*}
d_\mu \left( {F_k \left( x \right),F_{k + M} \left( x \right)}
\right) \leq \sum\limits_{i = k}^{k + M - 1} {d_\mu \left( {F_i
\left( x \right),F_{i + 1} \left( x \right)} \right)} \ .
\end{equation*}
From Lemma \ref{LemmaDistanceAtT} and by the assumption $w <
\frac{1}{8}$,
\begin{equation}\label{ConvergenceBound}
d_\mu  \left( {F_k \left( x \right),F_{k + M} \left( x \right)}
\right) \leq \sum\limits_{i = k}^{k + M - 1} {d_0 \left( {1 + 4w}
\right)} \left( {\frac{1} {2} + 4w} \right)^i  \leq d_0 \left( {1 +
4w} \right)\left( {\frac{1} {2} + 4w} \right)^k \frac{1}
{{\frac{1}{2} - 4w}} \ .
\end{equation}
We observe from (\ref{ConvergenceBound}) that for $w < \frac{1}{8}$,
$ \left\{ {F_k \left\{ x \right\}} \right\}_{k \in \mathbb{Z}_ +  }$
is a Cauchy sequence in the metric space $\left\{ {\mathfrak{J}_n
,d_\mu  } \right\}$, and consequently it is also a Cauchy sequence
in the metric space of Lebesgue measurable sets with the metric
$d_\mu\left(\cdot, \cdot\right)$, $\left\{ {\mathfrak{L}_n ,d_\mu  }
\right\}$. Since $\left\{ {\mathfrak{L}_n ,d_\mu  } \right\}$ is a
complete metric space, the sequence $ \left\{ {F_k \left\{ x
\right\}} \right\}_{k \in \mathbb{Z}_ +  }$ converges to $ F^\infty
\left( x \right) \in \mathfrak{L}_n$. The convergence is uniform in
$x$ due to (\ref{ConvergenceBound}), consequently $ F^\infty \left(
x \right)$ is continuous. \qed
\end{proof}
Next we derive results concerning approximation of H\"older
continuous SVFs by the 4-point subdivision scheme.
\begin{theorem}\label{FourPointKApprox} Let $G : \mathbb{R} \to \mathfrak{J}_n$ be $\nu$-H\"older continuous,
and let the initial sets be given by $F_i^0  = G\left( {\delta  +
ih} \right)$  $i \in \mathbb{Z}$ with $\delta \in [0,h)$ and $h >
0$. Then for $F_k\left(x\right)$ given by (\ref{Fk}),
\begin{equation}\label{FkApproximationOrder}
d_\mu \left( {G \left( x \right),F_k \left( x \right)} \right) \leq
Ch^\nu \ ,
\end{equation}
where $C = \left( {\frac{1}{{2^\nu  }} + \frac{1}{2} + \frac{{1 +
4w}}{{1/2 - 4w}}} \right)H$ and $H$ is the H\"older constant of $G$.

\end{theorem}
\begin{proof}
Without loss of generality assume that $\delta = 0$. Let $x$ be such
that $ih \leq x \leq  \left(i + \frac{1}{2}\right)h$. From the
triangle inequality,
\begin{equation*}
d_\mu \left( {G\left( x \right),F_k \left( x \right)} \right) \leq
d_\mu \left( {G\left( x \right),F_i } \right) + d_\mu \left( {F_i
,F_0 \left( x \right)} \right) + d_\mu \left( {F_0 \left( x
\right),F_k \left( x    \right)} \right) \ .
\end{equation*}
Using the H\"older continuity of $G$, the metric property of the
measure average and (\ref{ConvergenceBound}) we get,
\begin{equation}\label{SomeMoreConvergence}
d_\mu  \left( {G\left( x \right),F_k \left( x \right)} \right)
\leqslant H\left( {\frac{1} {2}h} \right)^\nu   + \frac{1} {2}d_0  +
\frac{{1 + 4w}} {{\frac{1} {2} - 4w}}d_0  \ .
\end{equation}
Now $d_0 \leq Hh^\nu$ with $H$ the H\"older constant of $G$, and
therefore (\ref{SomeMoreConvergence}) implies
(\ref{FkApproximationOrder}). \qed
\end{proof}
\begin{corollary}\label{FourPointInfApprox}Under the assumptions of Theorem
\ref{FourPointKApprox}, the distance between the original set-valued
function $G\left(x\right)$ and the limit set-valued function
$F^\infty \left( x \right)$  is bounded by,
\begin{equation*}
\mathop {\max d_\mu \left( {F^\infty  \left( x \right),G\left( x
\right)} \right)}  \leq C h^\nu \ ,
\end{equation*}
with $C$ as in Theorem \ref{FourPointKApprox} .
\end{corollary}
Since in extrapolation the measure property does not hold, there is
no result for the 4-point scheme similar to that in Corollary
\ref{MeasureTransition} for spline subdivision schemes.

\section{Computational examples}\label{sectionComputation}
The algorithms for the computation of the measure average along with
implementation details and the numerical evaluation of our methods
applied to the reconstruction of 3D objects from cross-sections are
given in \cite{kels2011reconstruction}. Here we provide several
computational examples illustrating the theory presented in the
previous sections.

In our examples we use the adaptations to sets with the measure
average of the piecewise linear interpolation (\emph{c.f.} Section
\ref{sectionMeasureAverage}), Chaikin subdivision scheme
(\emph{c.f.} Section \ref{sectionSplineSubdivision}) and the 4-point
subdivision scheme  (\emph{c.f.} Section \ref{section4pointScheme}).
Chaikin scheme is implemented using the Lane-Riesenfeld algorithm as
described in Section \ref{sectionSplineSubdivision}.

While the subdivision schemes discussed so far are defined over the
whole real axis, computations require to deal with a finite number
of sets, ant therefore to consider boundary rules for these schemes.
Boundary rules for the real-valued subdivision schemes are discussed
in (\cite{sabin2010analysis}, Chapter 32).

Assume that at the k-th level of the subdivision, we have the sets
$F_0^k,...,F_n^k$, assigned to equidistant points $x_0^k,...,x_n^k$.
For Chaikin scheme, the refinement rules at the boundaries are,
\begin{equation*}
\begin{array}{*{20}c}
   {F_0^{k + 1}  = F_0^k ,} & \qquad \qquad{F_1^{k + 1}  = \frac{1}
{2}F_0^k  \bigoplus \frac{1}
{2}F_1^k ,}   \\\\

   {F_{2n - 2}^{k + 1}  = \frac{1}
{2}F_{n - 1}^k  \bigoplus \frac{1}
{2}F_n^k ,} & \qquad \qquad{F_{2n - 1}^{k + 1}  = F_n^k ,}\\
\end{array}
\end{equation*}
and ${F_{i}^k }$, $i = 2,...,2n-3$ are computed using relations
(\ref{SetSplineRule1})-(\ref{SetSplineRule3}). With these boundary
rules the limit of the Chaikin scheme interpolates the sets at the
endpoints. For the 4-point scheme, the modified refinement rules at
the boundaries are,
\[
\begin{array}{*{20}c}
   {F_1^{k + 1}  = \frac{1}
{2}F_0^k  \bigoplus \frac{1} {2}F_1^k ;} & \qquad \qquad{F_{2n -
1}^{k + 1}  = \frac{1} {2}F_{n - 1}^k  \bigoplus \frac{1}
{2}F_n^k } , \\
\end{array}
\]
while $F_i^{k + 1}$,  $i = 0,2,3,...,2n-3,2n-2,2n$ are obtained
using relations
(\ref{RefinementRuleForSets1})-(\ref{RefinementRuleForSets2}). In
both schemes, the refined sets at level $k+1$ are assigned to
points, which are obtained from $x_0^k,...,x_n^k$  by applying rules
analogous to those applied to the sets.

In the first example we apply the three aforementioned methods to a
collection of eight sets $F_0 ...,F_7 $ with varying topologies,
located on equidistant parallel planes. The results are visualized
in Figures \ref{fig:PeiceWiseInterpolation} -
\ref{fig:RecontructionProcess4Point}. Note that in all examples only
boundaries of the sets are depicted. It can be observed from Figure
\ref{fig:PeiceWiseInterpolation}, that the piecewise interpolation
indeed passes through the original sets, but the transitions between
pairs of original consecutive sets are noticeable. Figure
\ref{fig:RecontructionProcessSpline} demonstrates the smoothing
effect of Chaikin subdivision, however the limit function does not
pass through the original sets. Finally, the limit of the 4-point
subdivision scheme in Figure \ref{fig:RecontructionProcess4Point}
has a certain smoothing effect, and yet it passes through the
original sets. The above geometric behavior of the three set-valued
methods is in analogy with the well known features of the
corresponding real-valued methods.

In Figure \ref{fig:Measure1} we plot $\mu
\left(F^{[i]}\left(x\right)\right)$, where $F^{[i]}\left(x\right)$,
$i = 1,2,3$, are the SVFs in Figures
\ref{fig:PeiceWiseInterpolation} -
\ref{fig:RecontructionProcess4Point} respectively. In accordance
with Corollaries \ref{GeneralMeasureTransition} and
\ref{MeasureTransition}, $\mu \left( {F^{[1]} \left( x \right)}
\right)$ and $\mu \left( {F^{[2]} \left( x \right)} \right)$ are
equal to the piecewise linear interpolant and the limit of Chaikin
subdivision applied to $\mu \left( {F_i } \right)$, $i = 0,...,7$.
While a result similar to Corollary \ref{MeasureTransition} does not
hold for the 4-point scheme, $\mu \left( {F_3 \left( x \right)}
\right)$ appears as a smooth curve interpolating the data $\mu
\left( {F_0 } \right),...,\mu \left( {F_7 } \right)$.

\begin{figure}[h!!!!!]
\begin{center}
\begin{tabular} {c c c}
\emph{a} & \emph{b} & \emph{c} \tabularnewline
\includegraphics[scale=0.16]{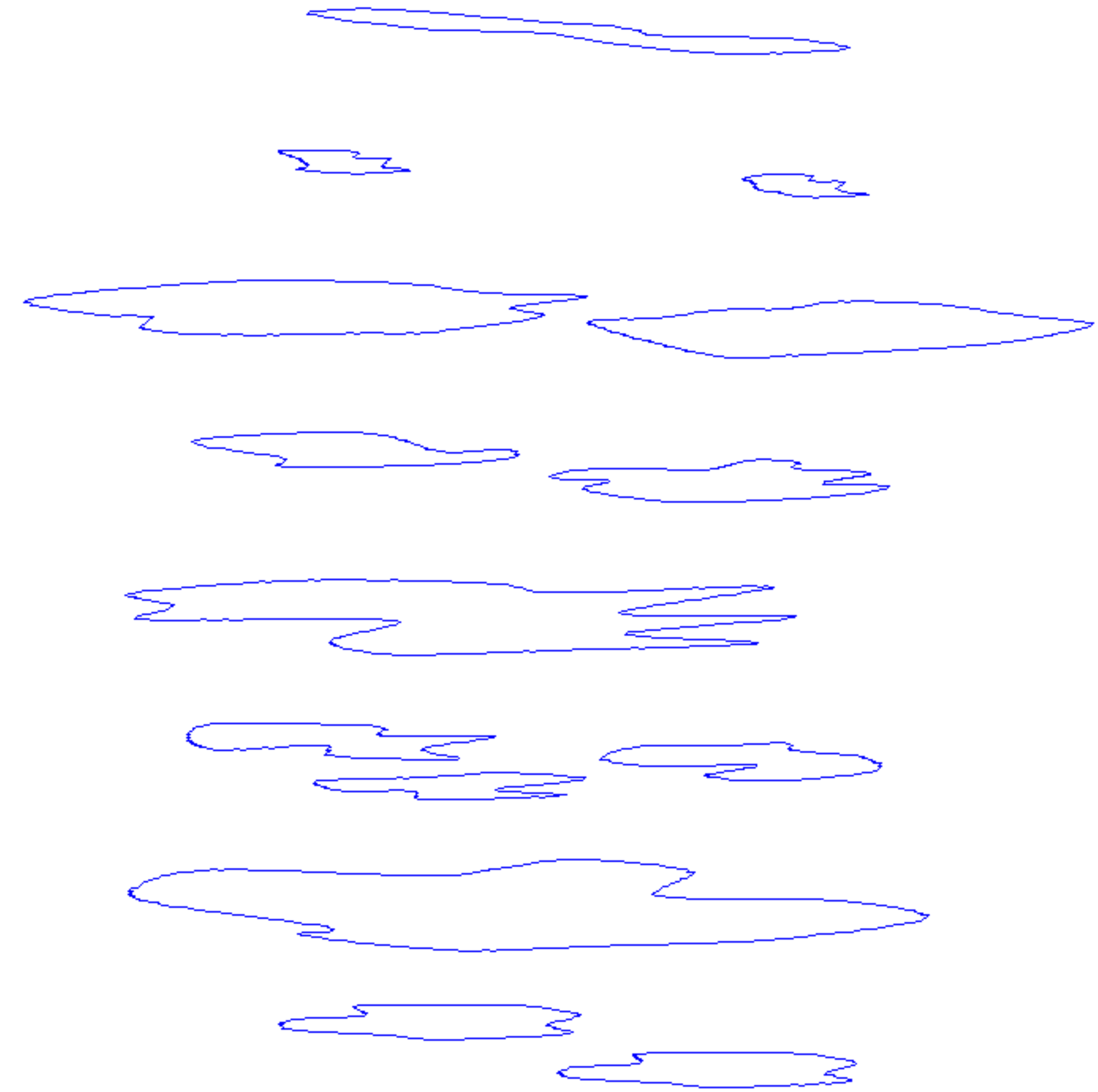} &
\includegraphics[scale=0.16]{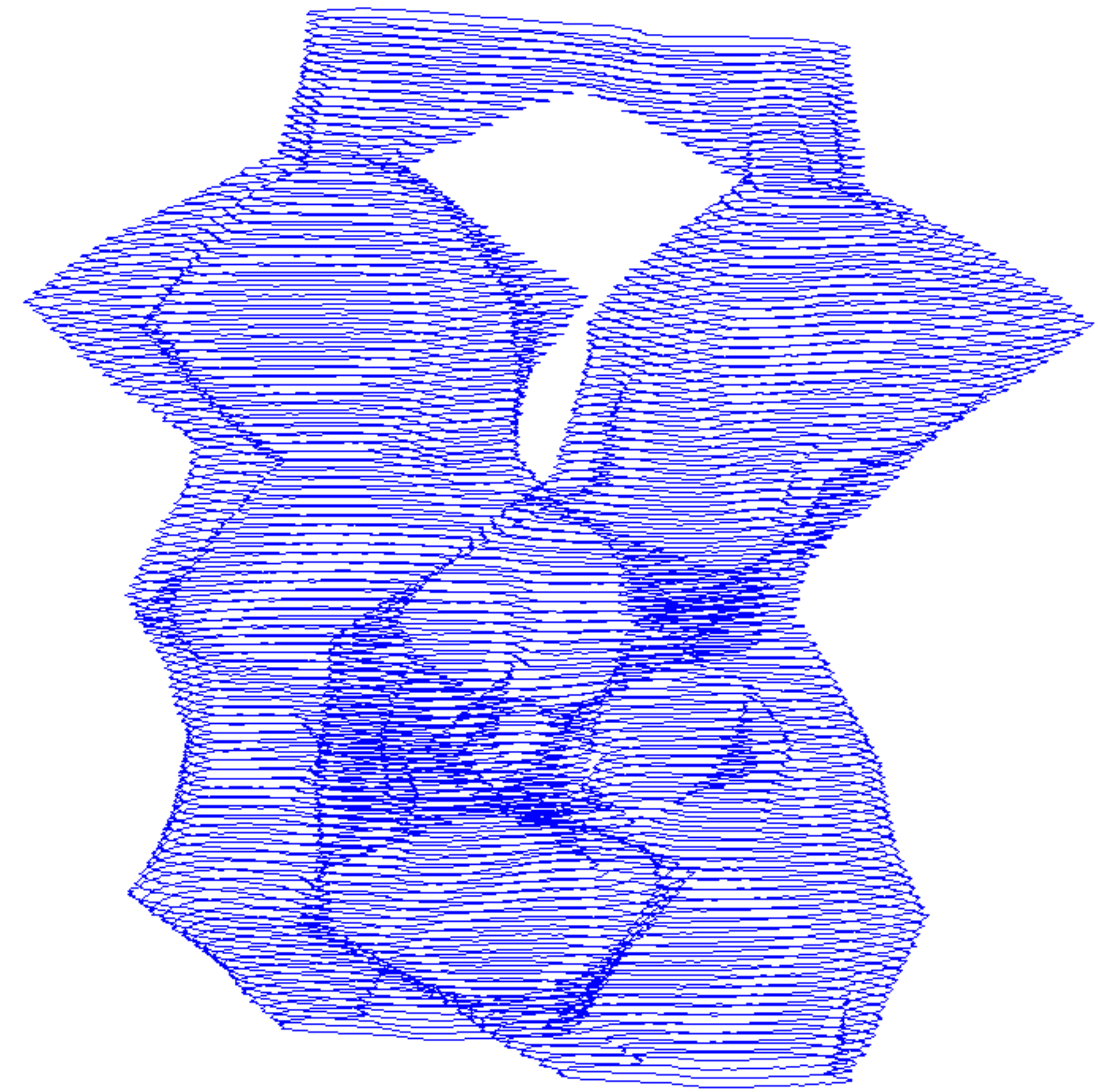} &
\includegraphics[scale=0.16]{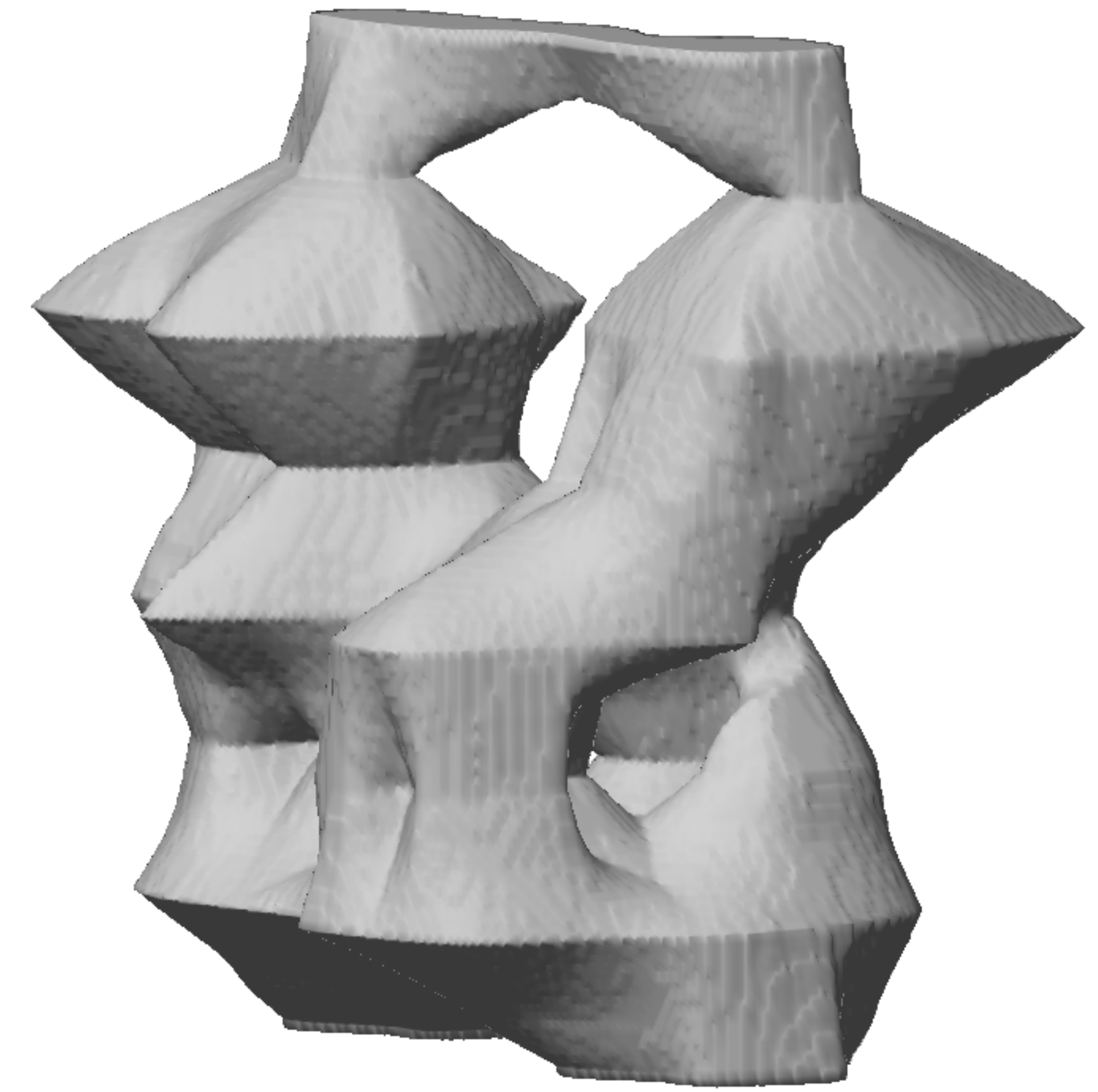}
\end{tabular}
\end{center}
\caption{Piecewise interpolation of sets: a. the initial sets; b.
15 sets introduced between each pair of consecutive sets using the
measure average; c. visualization of the resulting SVF as a 3D
object.} \label{fig:PeiceWiseInterpolation}
\end{figure}

\begin{figure}[h!!!!!]
\begin{center}
\begin{tabular}{c   c  c}
\emph{a} & \emph{b} & \emph{c} \tabularnewline
\includegraphics[scale=0.16]{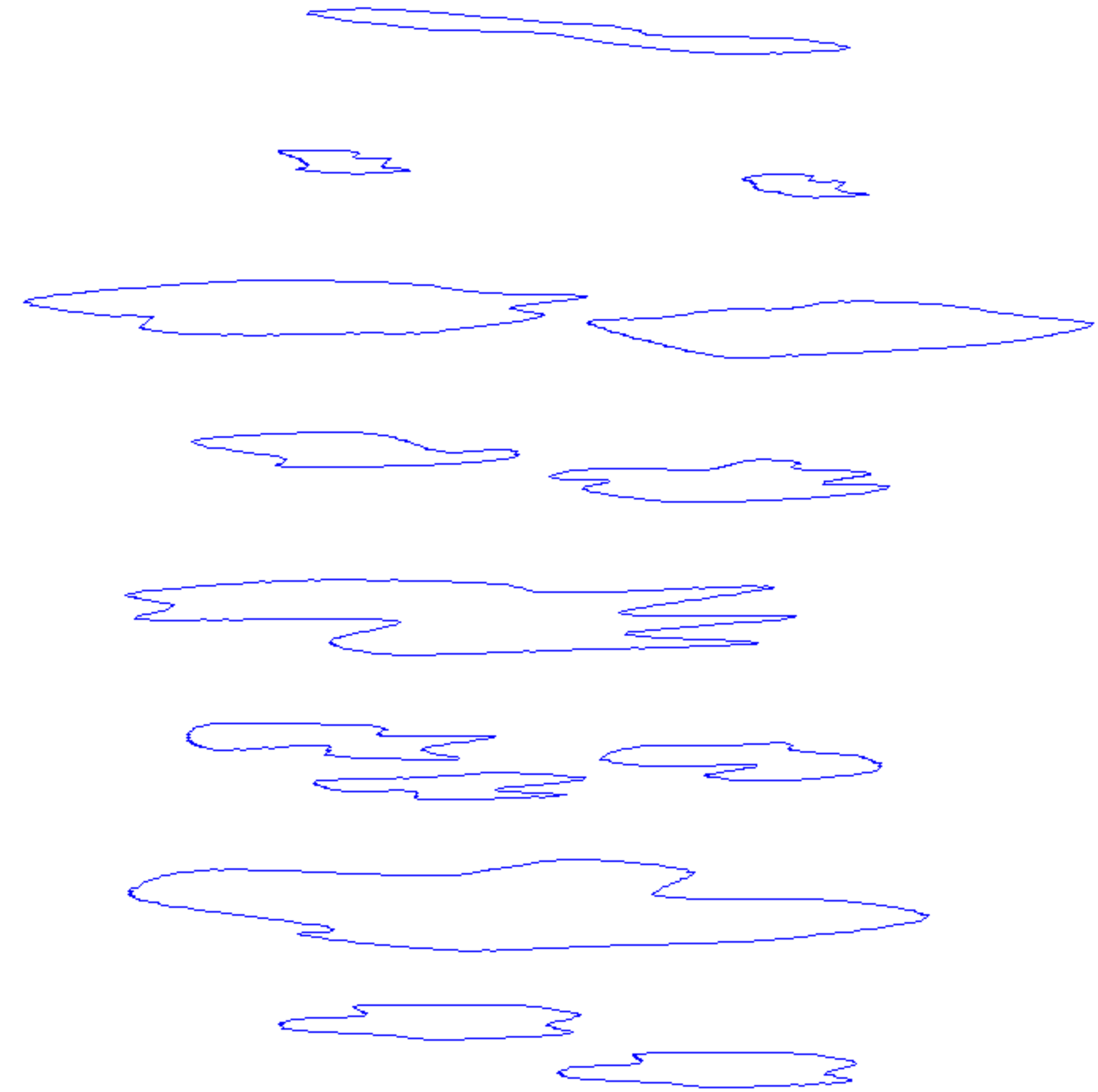} & \includegraphics[scale=0.16]{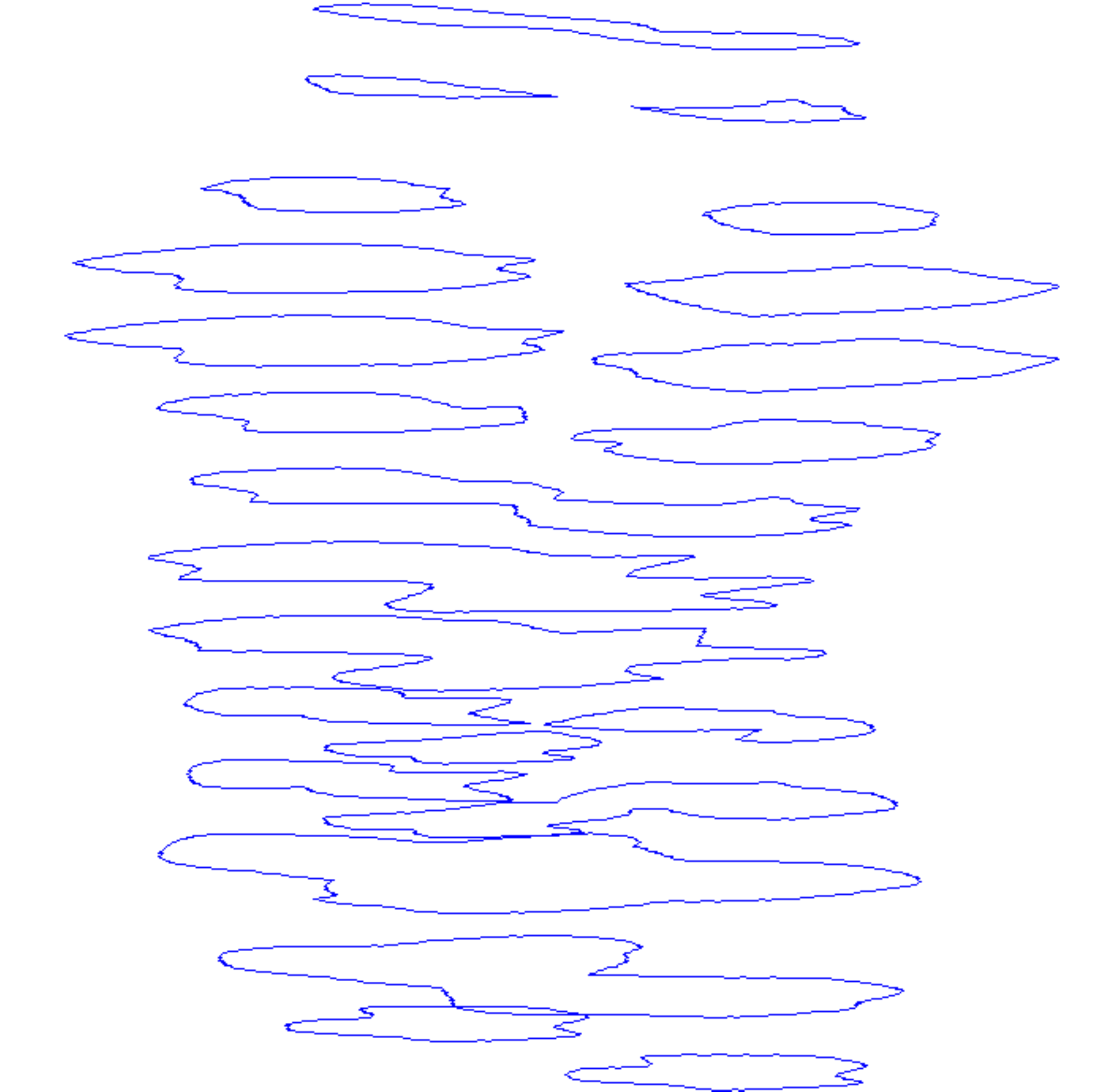} & \includegraphics[scale=0.16]{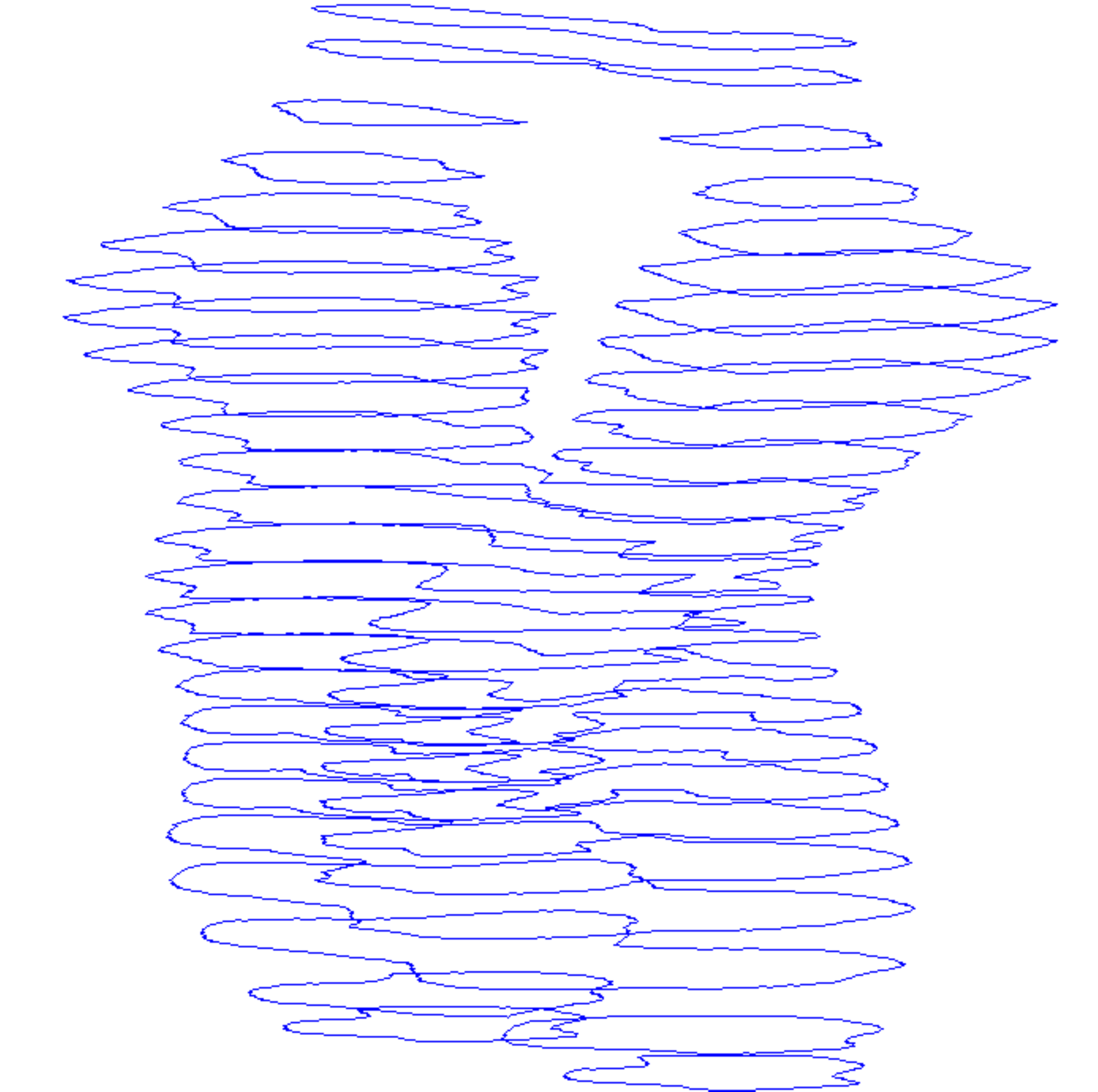}
 \tabularnewline
\emph{d} & \emph{e} & \emph{f} \tabularnewline
\includegraphics[scale=0.16]{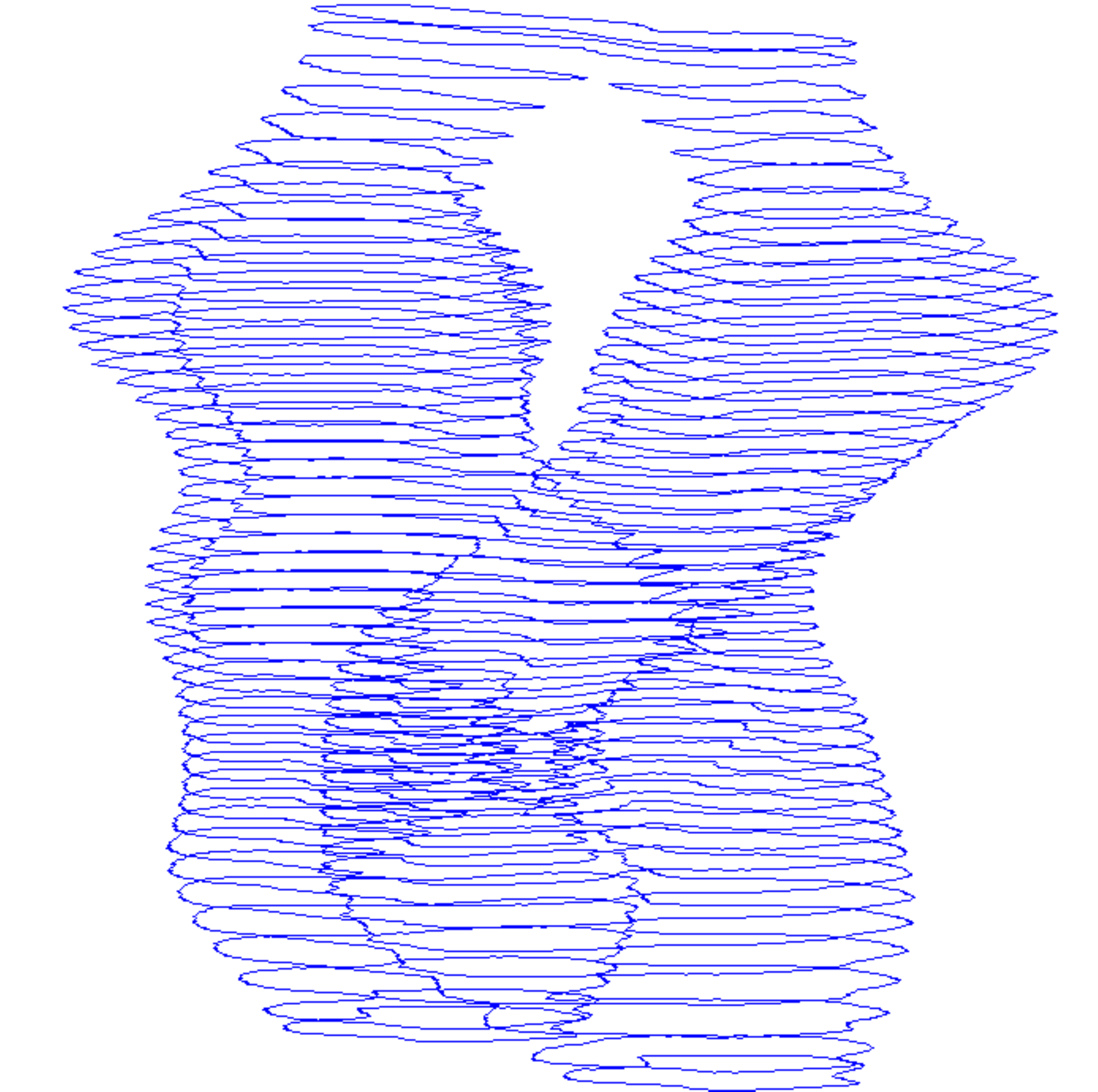} & \includegraphics[scale=0.16]{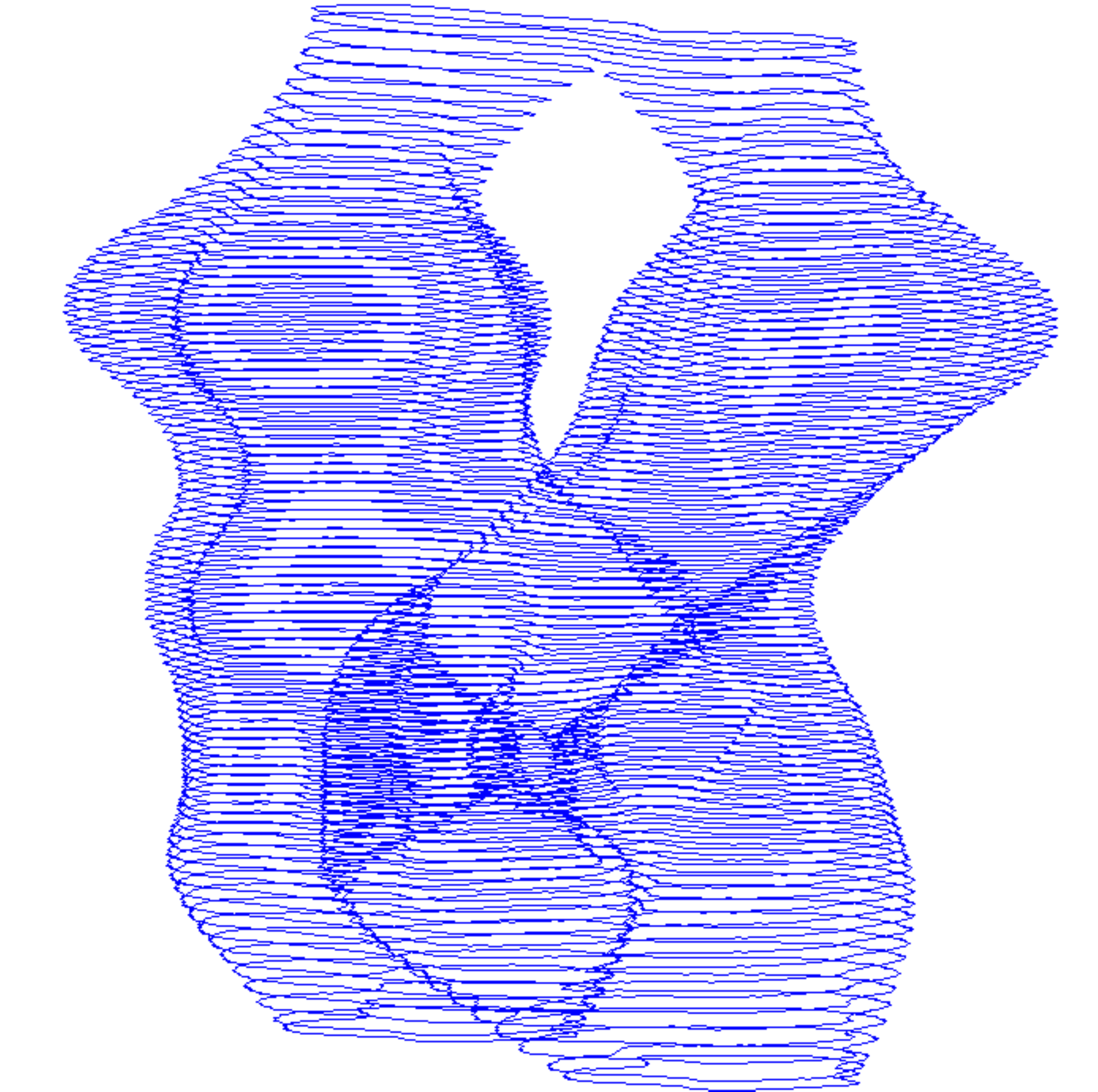} & \includegraphics[scale=0.16]{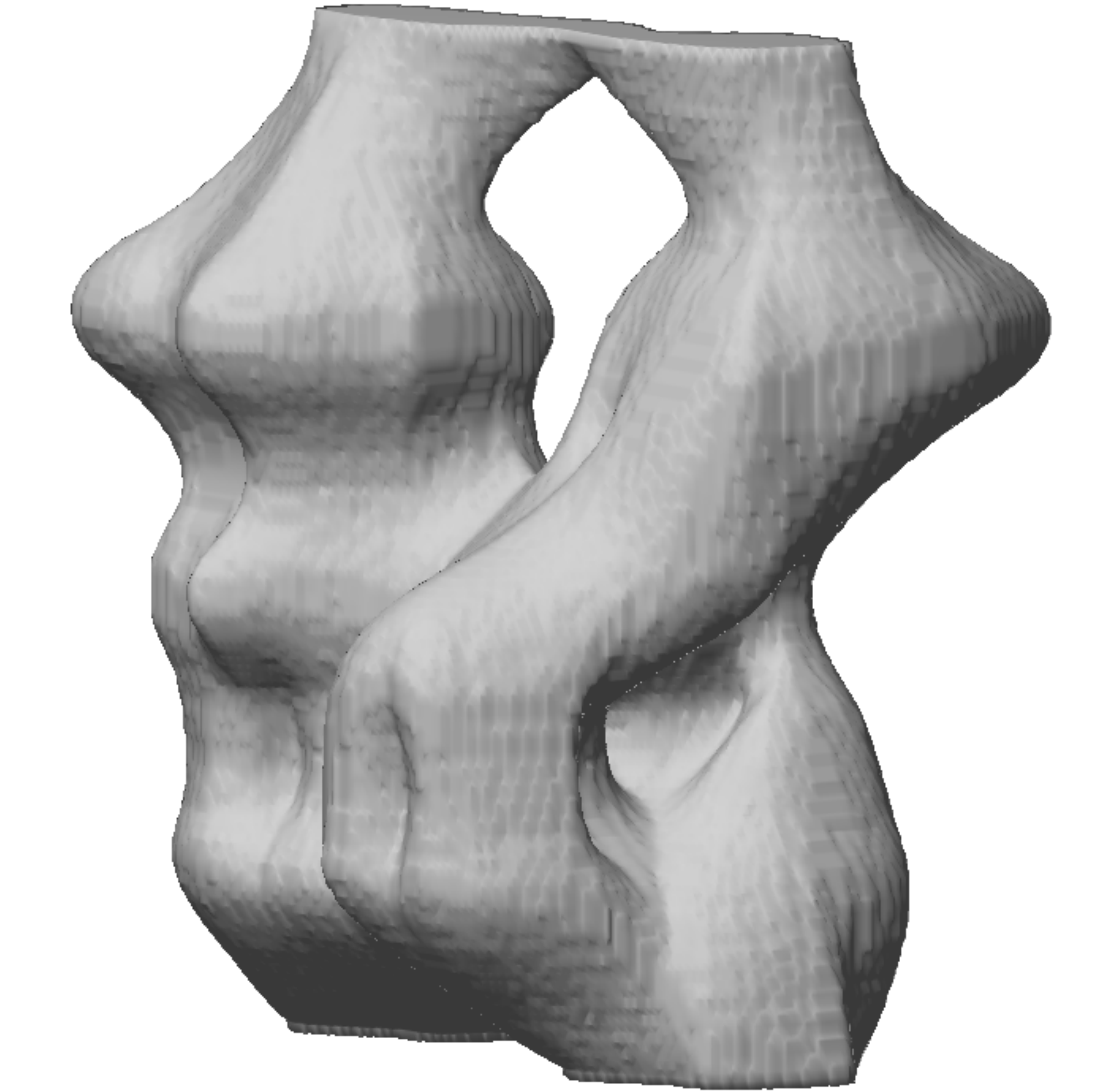}
\tabularnewline
\end{tabular}
\caption{Refinement of sets with Chaikin subdivision scheme: a. the
initial sets; b,c,d,e. the refined sets after 1,2,3,4 subdivision
steps respectively; f. visualization of the resulting SVF as a 3D
object. } \label{fig:RecontructionProcessSpline}
\end{center}
\end{figure}

\begin{figure}[h!!!!]
\begin{center}
\begin{tabular}{c   c  c}
\emph{a} & \emph{b} & \emph{c} \tabularnewline
\includegraphics[scale=0.16]{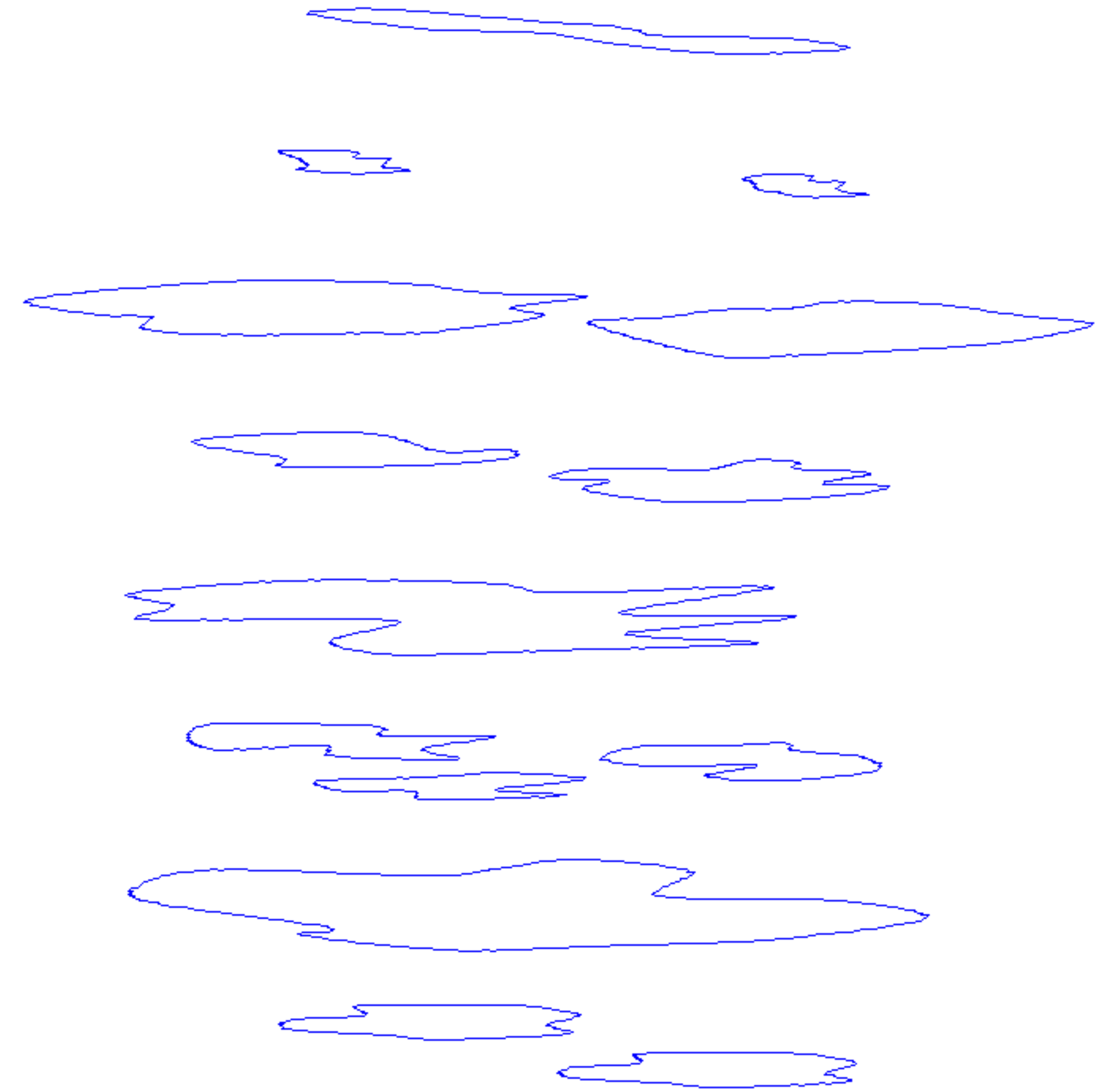} & \includegraphics[scale=0.16]{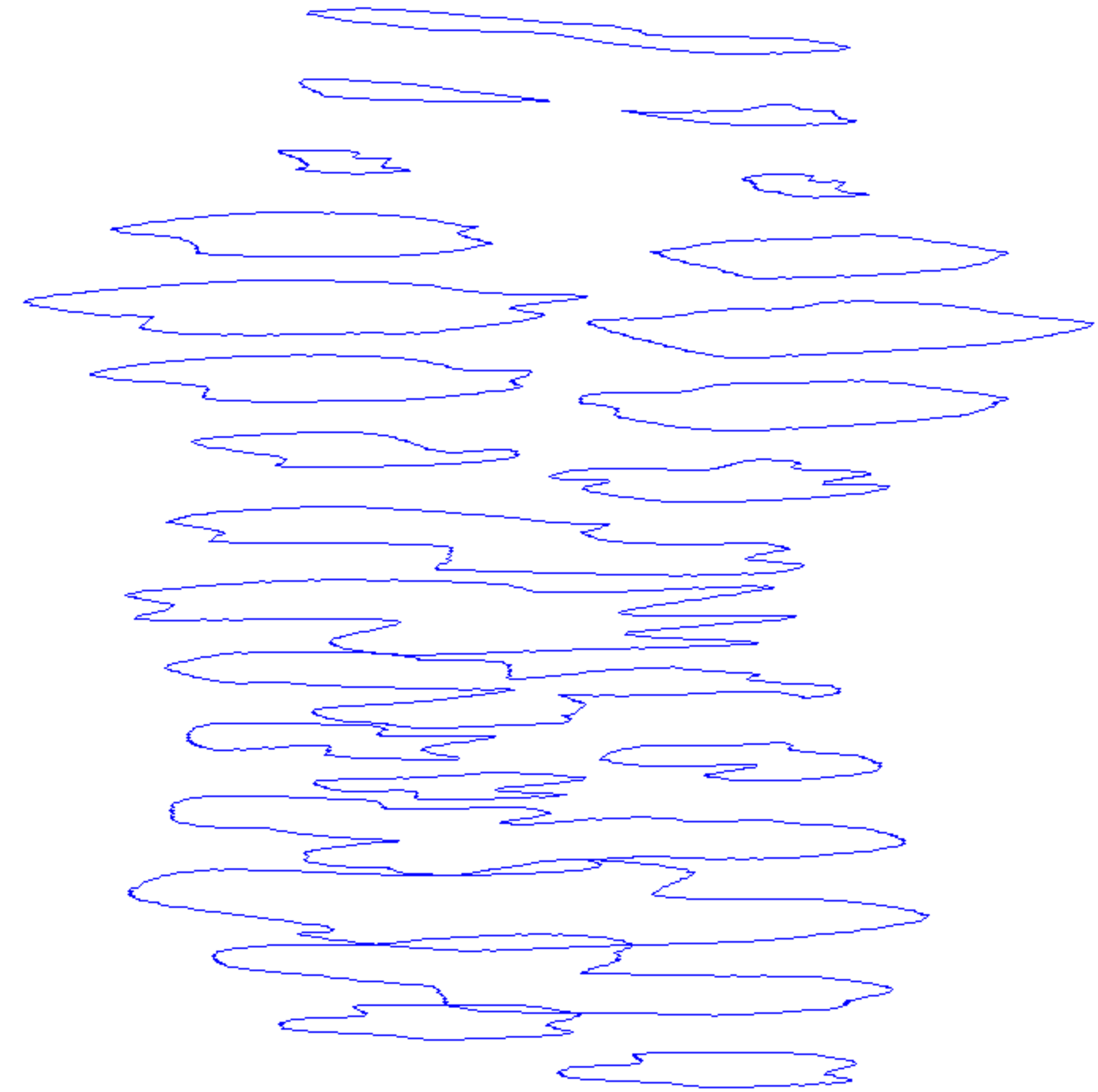} & \includegraphics[scale=0.16]{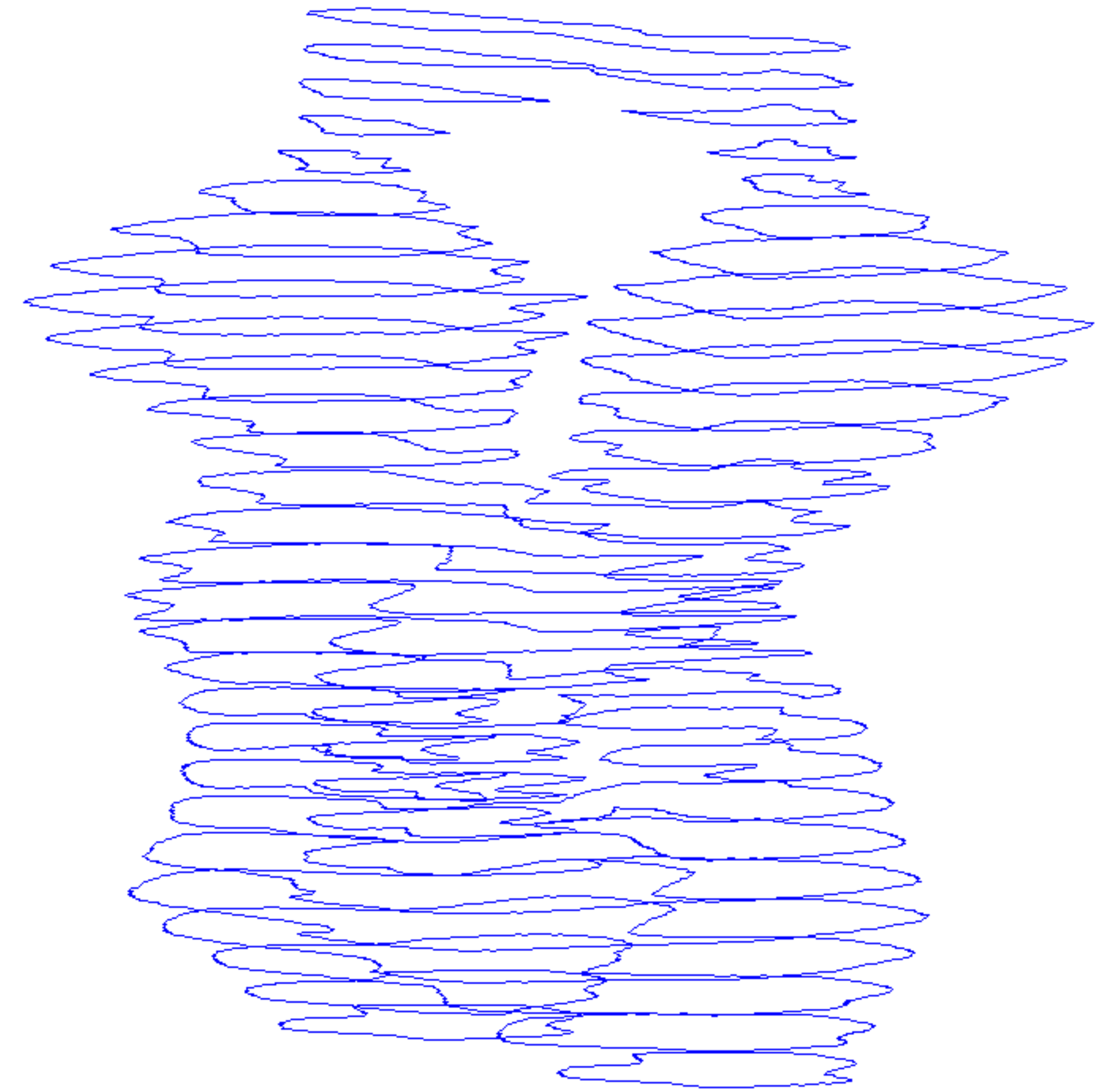}
 \tabularnewline
\emph{d} & \emph{e} & \emph{f} \tabularnewline
\includegraphics[scale=0.16]{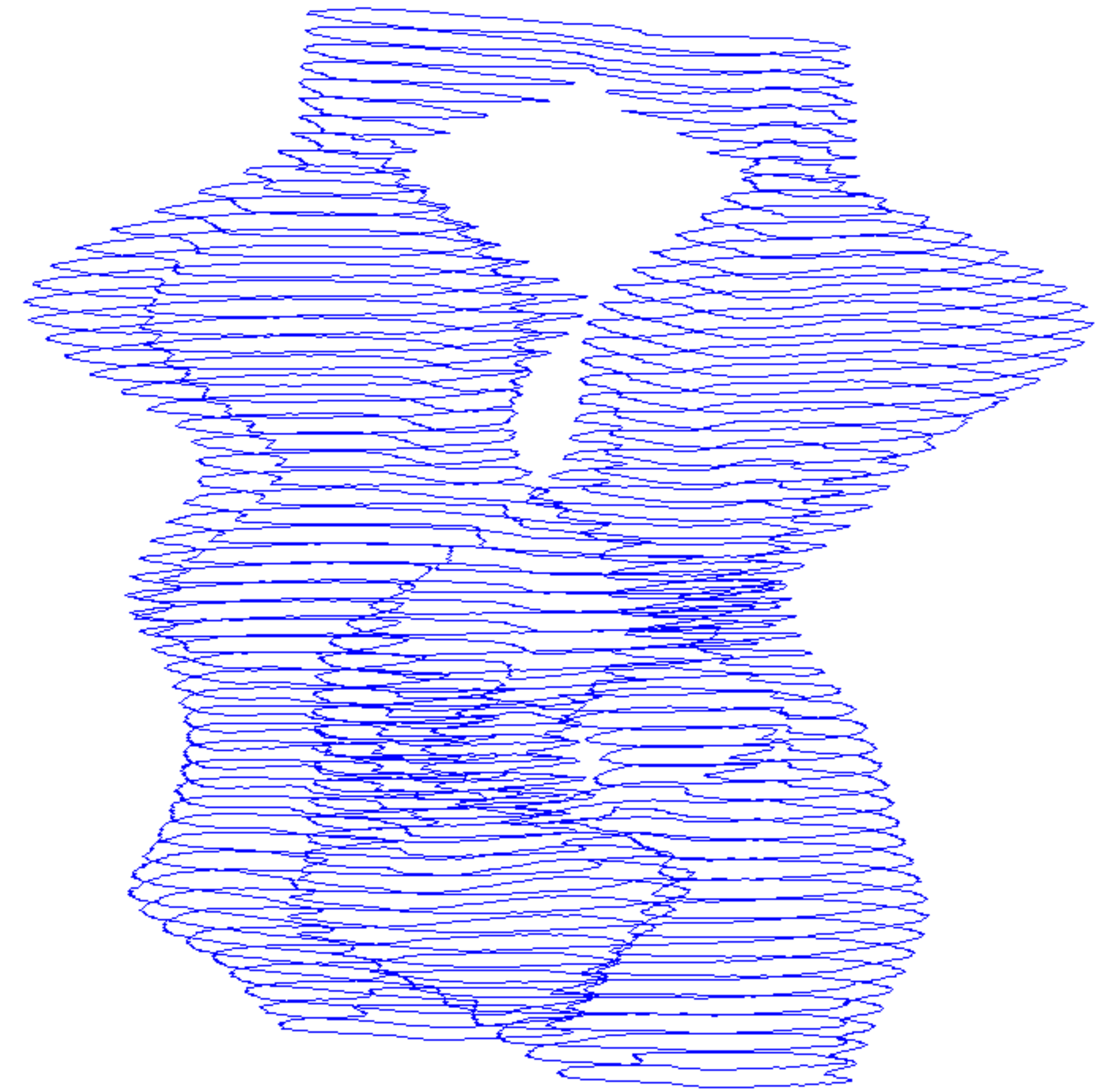} & \includegraphics[scale=0.16]{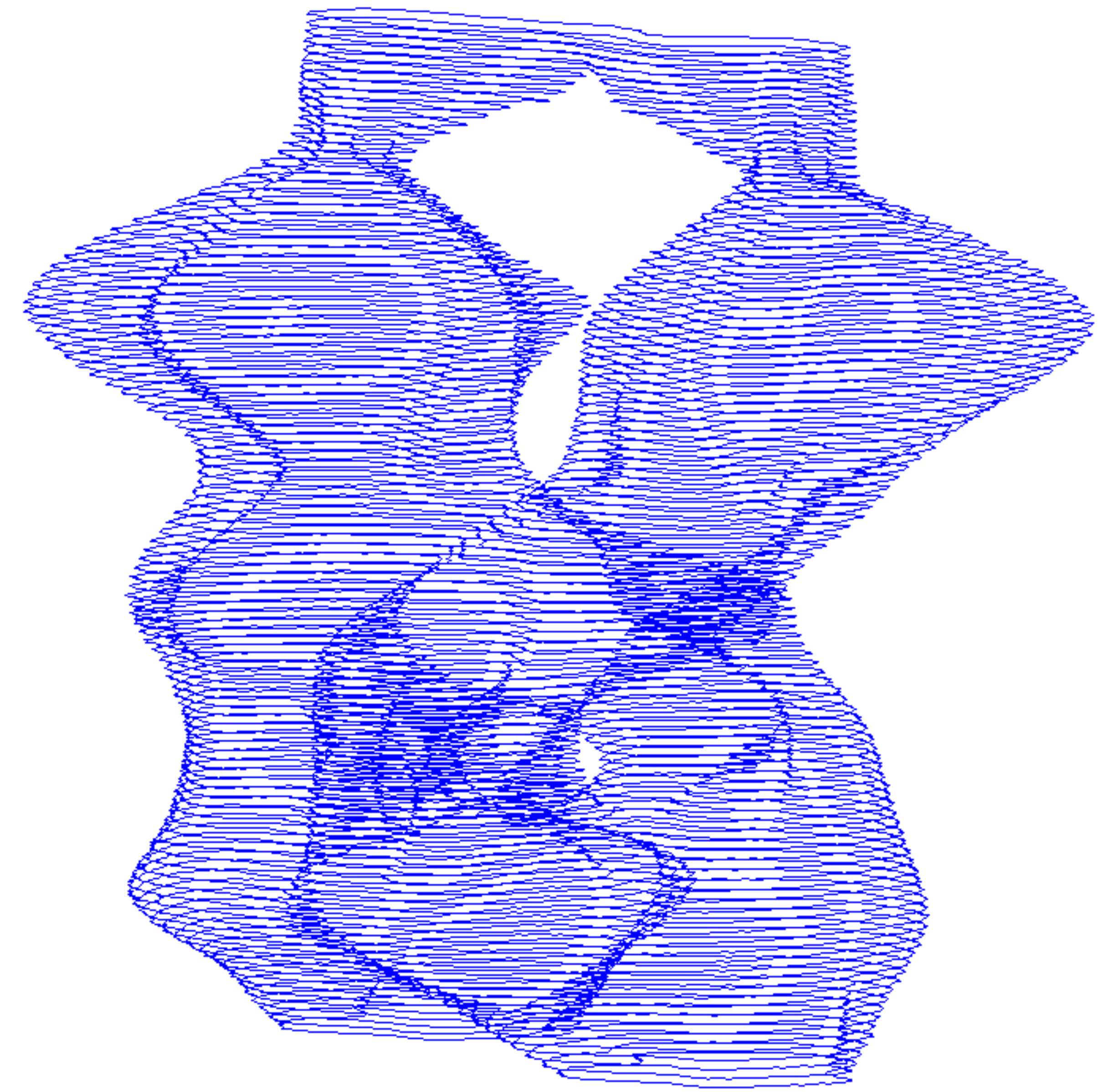} & \includegraphics[scale=0.16]{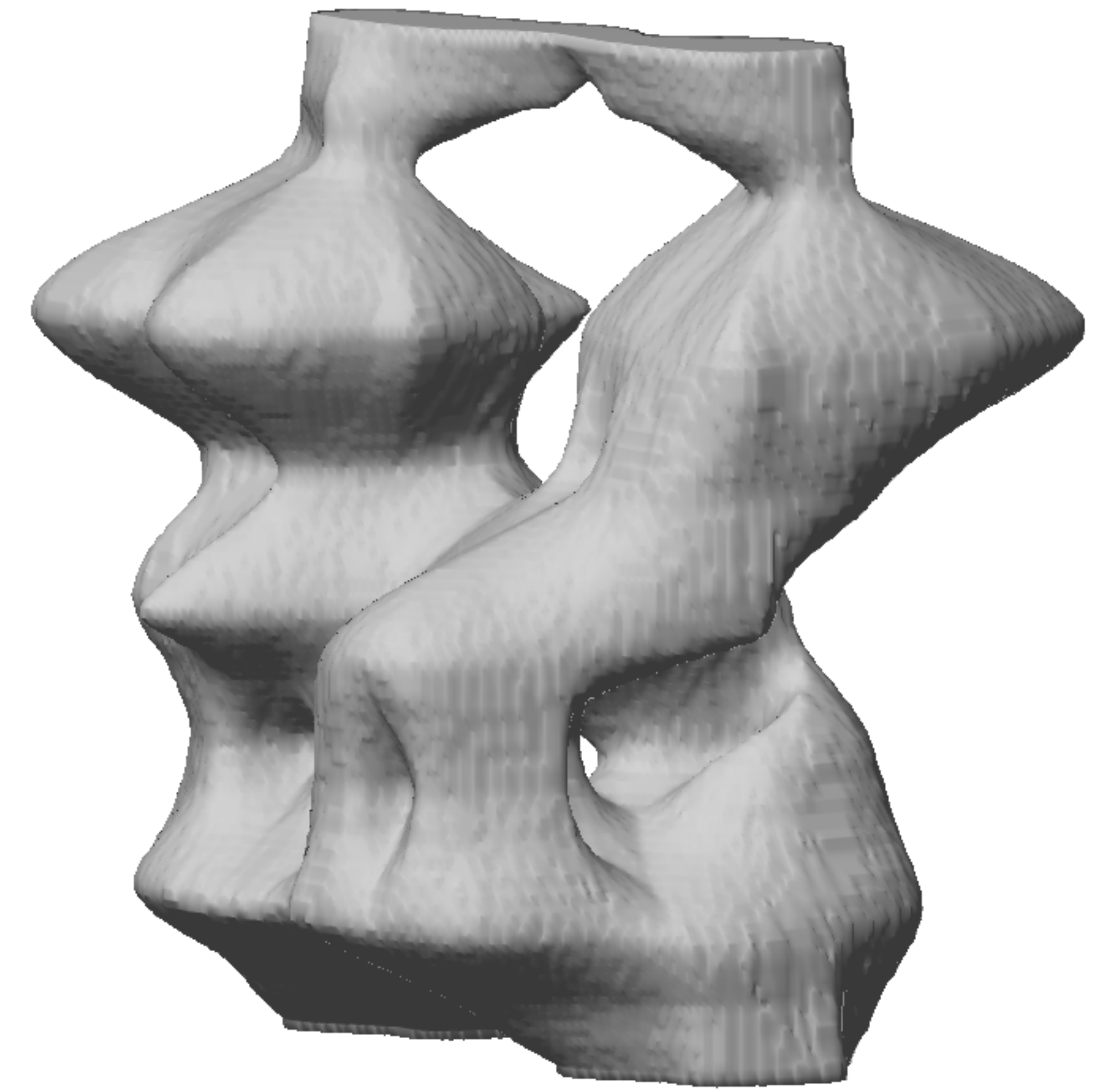}
\tabularnewline
\end{tabular}
\caption{Refinement of sets with the 4-point subdivision scheme: a.
the initial sets; b,c,d,e. the refined sets after 1,2,3,4
subdivision steps respectively; f. visualization of the resulting
SVF as a 3D object. } \label{fig:RecontructionProcess4Point}
\end{center}
\end{figure}

\begin{figure}[h!!!!!]
\begin{center}
\includegraphics[scale=0.6]{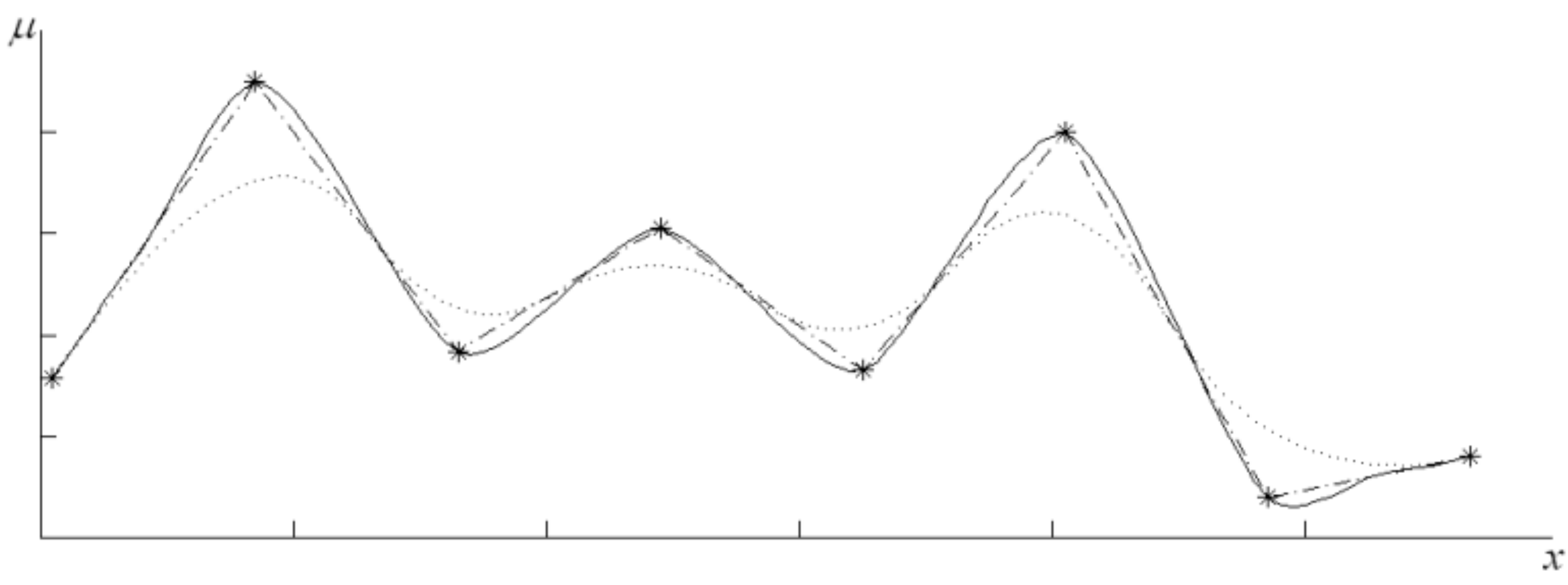}
\caption{$\mu \left( {F^{[i]} \left( x \right)} \right)$, $i =
1,2,3$, for the SVFs in Figures \ref{fig:PeiceWiseInterpolation},
\ref{fig:RecontructionProcessSpline},
\ref{fig:RecontructionProcess4Point}, depicted by the dash-dotted,
the dotted and the solid lines respectively. The measures of the
initial sets are denoted by $*$.} \label{fig:Measure1}
\end{center}
\end{figure}

Next we consider the piecewise interpolation and Chaikin subdivision
applied to monotone data $\widetilde F_0 ,...,\widetilde F_7 $,
$\widetilde F_i \supseteq \widetilde F_{i + 1}$, $i = 0,...,6$. In
view of the inclusion property of the measure average and Corollary
\ref{Monotonicity} the resulting SVFs are monotone in the set-valued
sense, as illustrated in Figure
\ref{fig:RecontructionProcessMonotone}. In Figure
\ref{fig:MeasureMonotone} we plot $\mu \left(\widetilde
F^{[i]}\left(x\right)\right)$, where $\widetilde
F^{[i]}\left(x\right)$, $i = 1,2$, are the SVFs obtained using
piecewise interpolation and Chaikin subdivision respectively. By
Corollary \ref{MeasureTransition}, $\mu \left(\widetilde
F^{[i]}\left(x\right)\right)$, $i = 1,2$ are the same as the
functions obtained by the application of piecewise linear
interpolation and Chaikin subdivision respectively to the initial
data $\mu \left( {\widetilde F_i } \right)$, $i = 0,...,7$. In
particular $\mu \left( {\widetilde F^{[2]}} \left(x\right)\right)$
has a continuous derivative, which in view of Corollary
\ref{SplineMetricDerivative} illustrates the continuity of the
metric velocity of ${\widetilde F^{[2]}} \left(x\right)$.

\begin{figure}
\begin{center}
\begin{tabular}{c   c  c}
\emph{a} & \emph{b} & \emph{c} \tabularnewline
\includegraphics[scale=0.4]{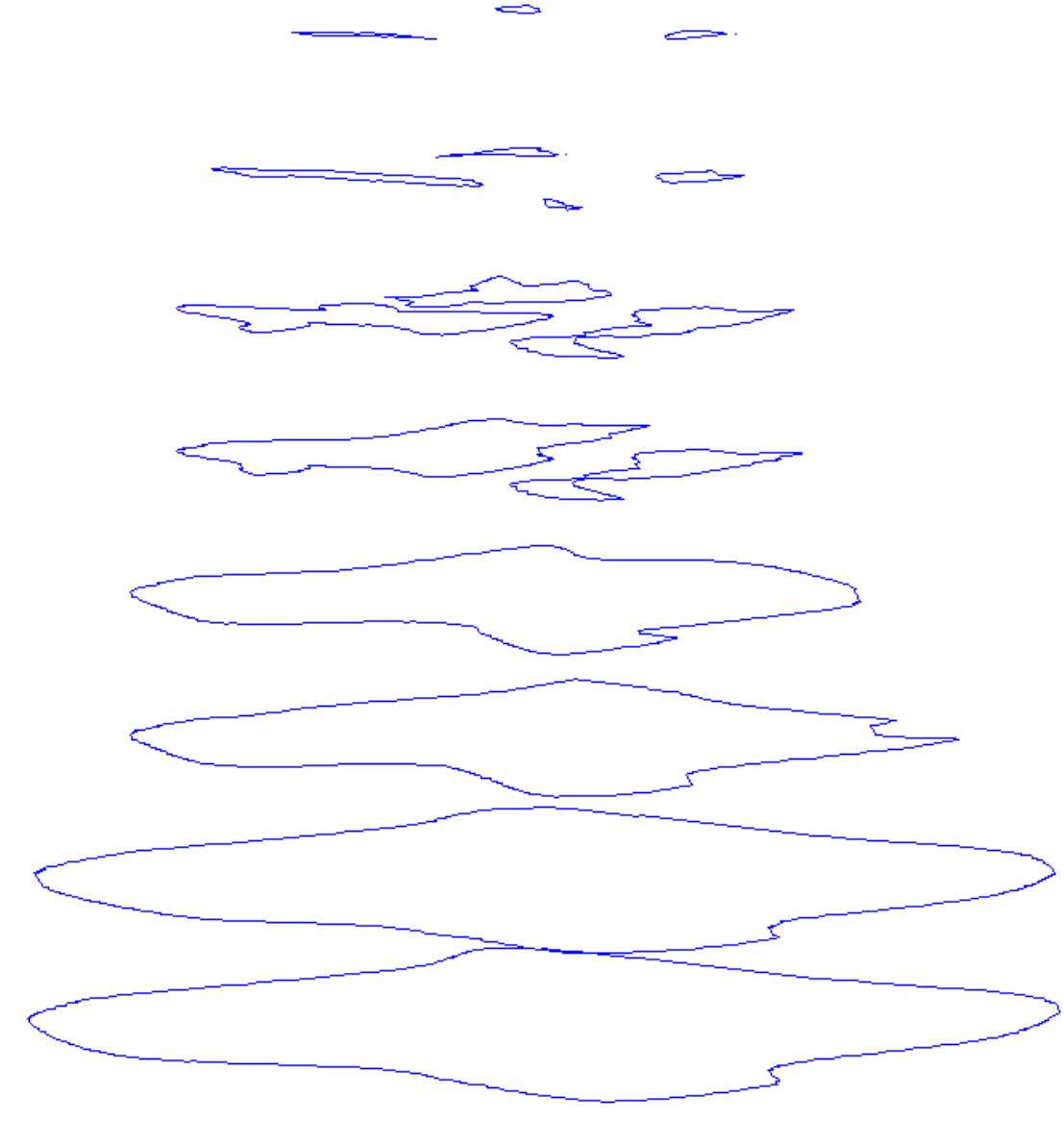} & \includegraphics[scale=0.4]{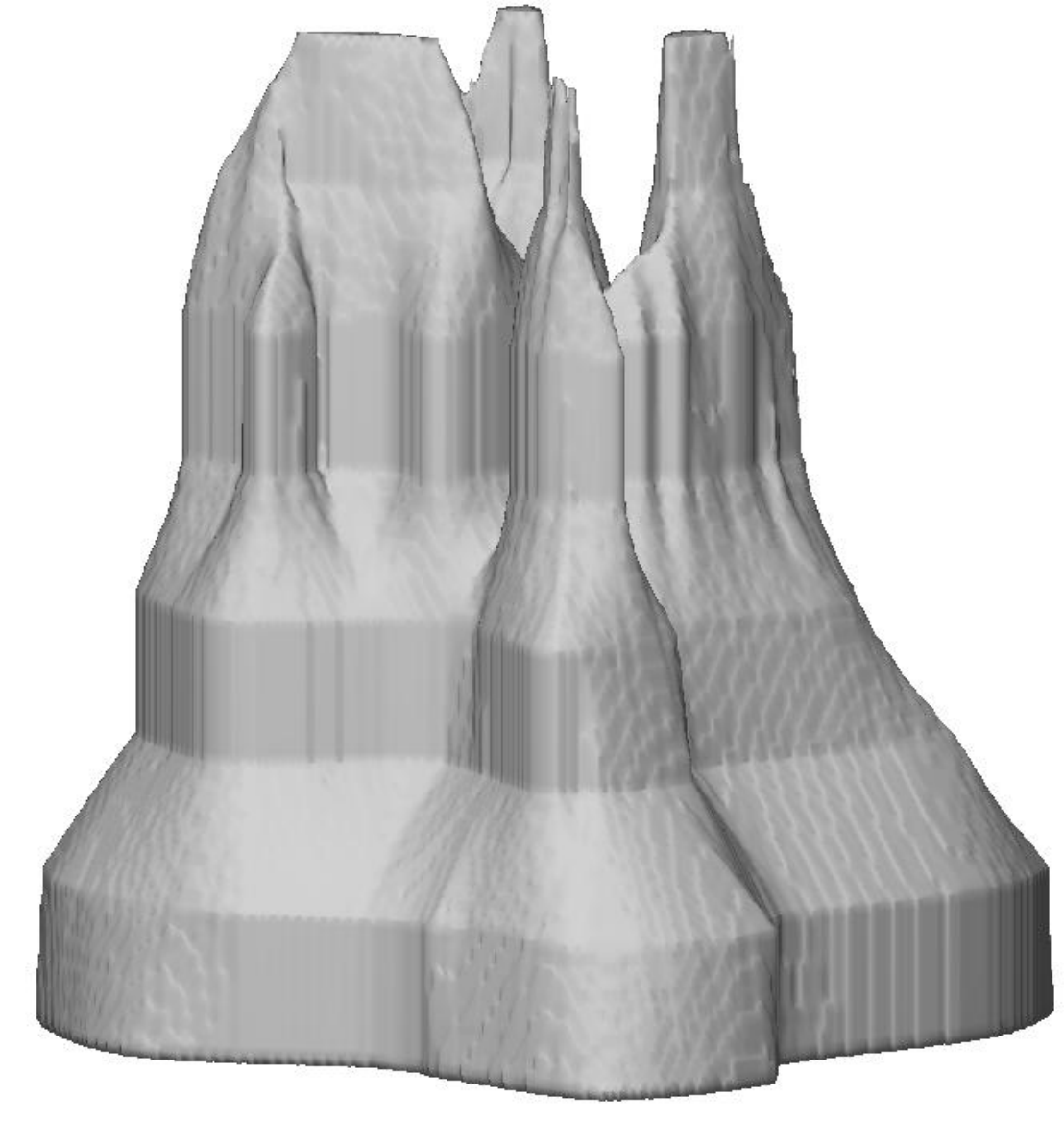} & \includegraphics[scale=0.4]{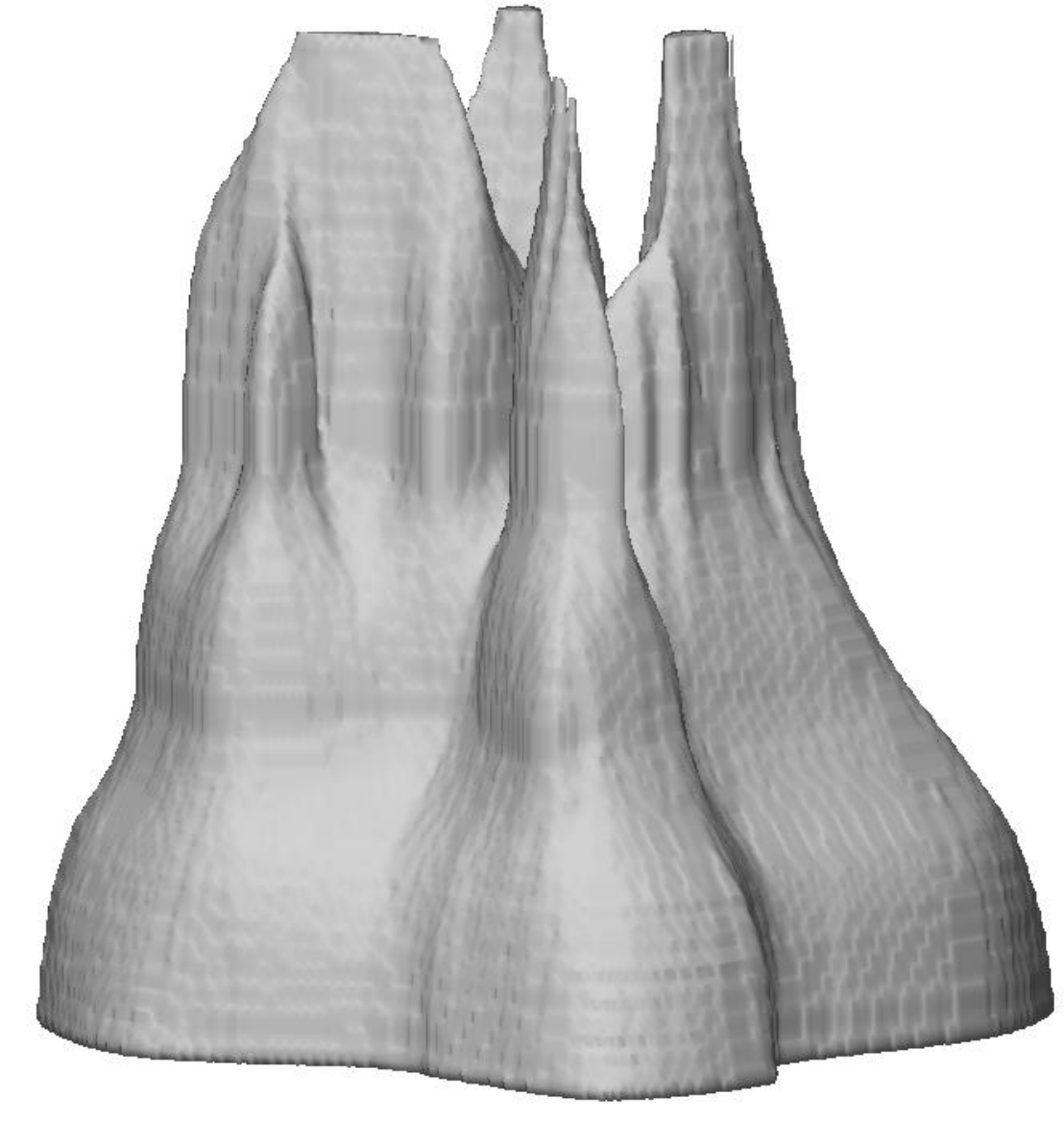}
\end{tabular}
\caption{The piecewise interpolation and Chaikin subdivision applied
to monotone set-valued data: a. the original sets b. piecewise
interpolation c. Chaikin subdivision}
\label{fig:RecontructionProcessMonotone}
\end{center}
\end{figure}

\begin{figure}
\begin{center}
\includegraphics[scale=0.4]{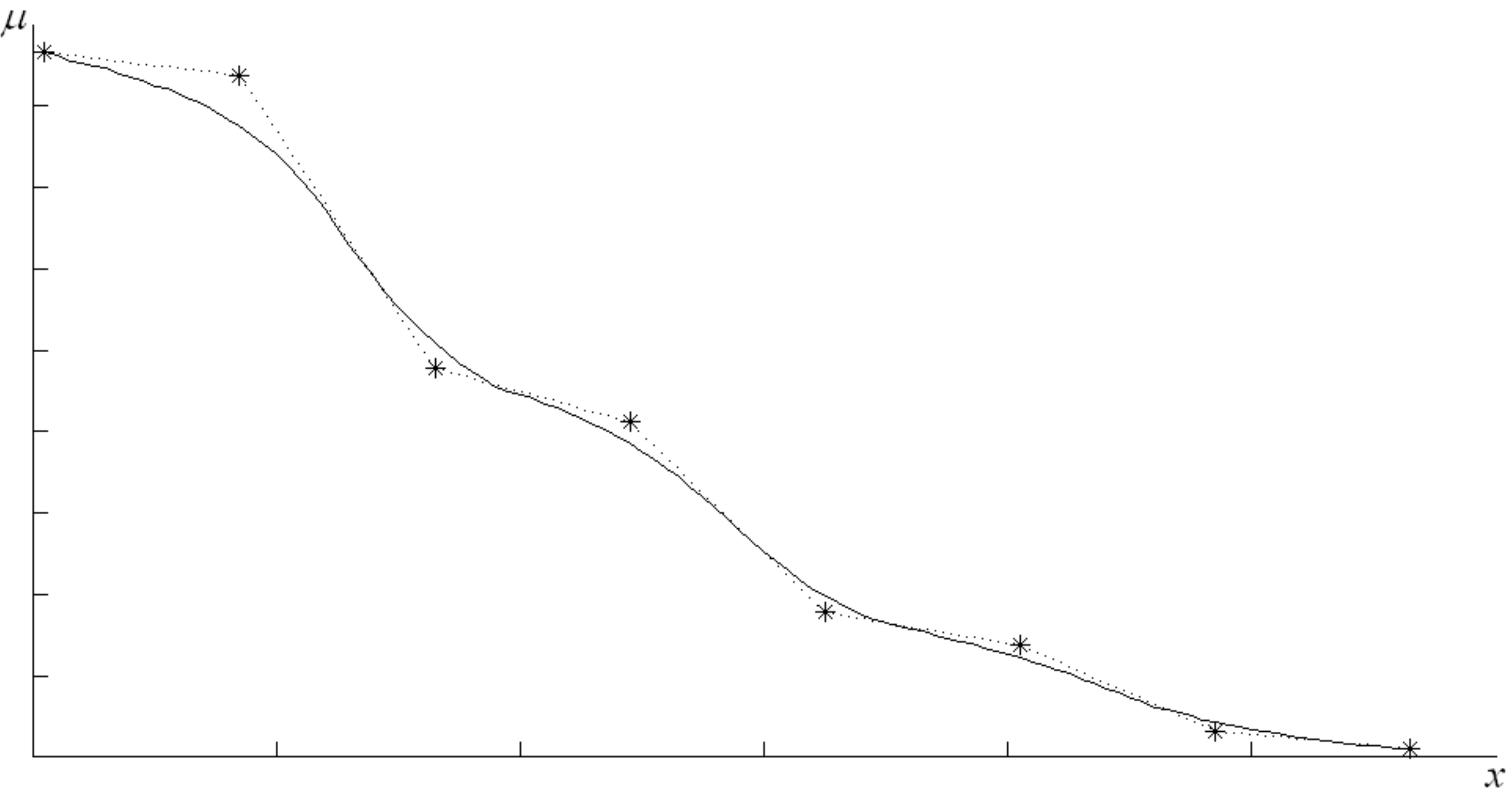}
\caption{$\mu \left( {\widetilde F^{[i]} \left( x \right)} \right)$,
$i = 1,2$, for the SVFs in Figure
\ref{fig:RecontructionProcessMonotone}b and Figure
\ref{fig:RecontructionProcessMonotone}c, depicted by the dotted and
the solid lines respectively. The measures of the initial sets are
denoted by $*$.} \label{fig:MeasureMonotone}
\end{center}
\end{figure}

\section{Extensions}\label{sectionExtensions}
We extend some of the results obtained in metric spaces of sets in
this work and in \cite{dyn2001spline} to general metric spaces
endowed with a binary average satisfying certain properties. Let
$\left\{ {X,d_X } \right\}$ be a metric space, and let $ \boxplus$
be an average on elements of $X$ defined for non-negative averaging
parameters ($\boxplus : [0,1] \times X \times X \to X$). Assume that
the average $\boxplus$ satisfies the interpolation property and the
metric property in List of Properties \ref{PropertiesList} with
$\mathfrak{J}_n$, $d_\mu$ and $\bigoplus$ replaced by $X$,$d_X$ and
$\boxplus$ respectively. Then:
\begin{enumerate}
\item A piecewise interpolant based on $\boxplus$ can be defined as in (\ref{PieceApproxEq}),
leading to an approximation result analogous to Corollary
\ref{PieceApprox}.

\item The spline subdivision schemes can be defined using the Lane-Riesenfeld algorithm.
Such an adaptation leads to convergence and approximation results
analogous to Theorems
\ref{SplineConvergence},\ref{SplineKApproximation} and Corollary
\ref{SplineInfApproximation}, with the limit of the subdivision in
the \emph{metric completion} of $\left\{ {X,d_X } \right\}$.

\item  Under the additional assumptions that the average $\boxplus$ is defined  also for averaging parameters outside
$[0,1]$ and satisfies the submetric property in List of Properties
\ref{PropertiesList},the 4-point subdivision scheme can be adapted
to the elements of $X$, as in relations
(\ref{RefinementRuleForSets1})-(\ref{RefinementRuleForSets2}).
Convergence and approximation results analogous to Theorems
\ref{FourPointConvergence}, \ref{FourPointKApprox} and Corollary
\ref{FourPointInfApprox} are obtained in a similar way.
\end{enumerate}

\bibliographystyle{abbrv}

\begin{thebibliography}{10}

\bibitem{albu2008morphology}
A.~Albu, T.~Beugeling, and D.~Laurendeau.
\newblock {A Morphology-Based Approach for Interslice Interpolation of
  Anatomical Slices From Volumetric Images}.
\newblock {\em IEEE Transactions on Biomedical Engineering}, 55(8), 2008.

\bibitem{artstein1989pla}
Z.~Artstein.
\newblock {Piecewise linear approximations of set-valued maps}.
\newblock {\em Journal of Approximation Theory}, 56(1):41--47, 1989.

\bibitem{baier2001differences}
R.~Baier and E.~Farkhi.
\newblock {Differences of convex compact sets in the space of directed sets.
  Part I: The space of directed sets}.
\newblock {\em Set-Valued Analysis}, 9(3):217--245, 2001.

\bibitem{baier2011set}
R.~Baier and G.~Perria.
\newblock {Set-valued Hermite interpolation}.
\newblock {\em Journal of Approximation Theory}, 2011.

\bibitem{barequet2007nonlinear}
G.~Barequet and A.~Vaxman.
\newblock {Nonlinear interpolation between slices}.
\newblock In {\em Proceedings of the 2007 ACM symposium on Solid and physical
  modeling}, pages 97--107. ACM, 2007.

\bibitem{bors2002binary}
A.~Bors, L.~Kechagias, and I.~Pitas.
\newblock {Binary morphological shape-based interpolation applied to 3-D tooth
  reconstruction}.
\newblock {\em IEEE transactions on medical imaging}, 21(2):100--108, 2002.

\bibitem{burago2001course}
D.~Burago, Y.~Burago, S.~Ivanov, and A.~M. Society.
\newblock {\em {A course in metric geometry}}.
\newblock American Mathematical Society, 2001.

\bibitem{deslauriers1989symmetric}
G.~Deslauriers and S.~Dubuc.
\newblock {Symmetric iterative interpolation processes}.
\newblock {\em Constructive approximation}, 5(1):49--68, 1989.

\bibitem{dyn2000spline}
N.~Dyn and E.~Farkhi.
\newblock {Spline subdivision schemes for convex compact sets}.
\newblock {\em Journal of Computational and Applied Mathematics},
  119(1-2):133--144, 2000.

\bibitem{dyn2001spline}
N.~Dyn and E.~Farkhi.
\newblock {Spline subdivision schemes for compact sets with metric averages}.
\newblock {\em Trends in Approximation Theory}, pages 93--102, 2001.

\bibitem{dyn2002spline}
N.~Dyn and E.~Farkhi.
\newblock {Spline Subdivision Schemes for Compact Sets. A Survey}.
\newblock {\em Serdica Math. J}, 28(4):349--360, 2002.

\bibitem{dyn2005set}
N.~Dyn and E.~Farkhi.
\newblock {Set-valued approximations with Minkowski averages--convergence and
  convexification rates}.
\newblock {\em Numerical Functional Analysis and Optimization}, 25(3):363--377,
  2005.

\bibitem{dynapproximation}
N.~Dyn, E.~Farkhi, and A.~Mokhov.
\newblock {Approximation of univariate set-valued functions-an overview}.
\newblock {\em Serdica Math. J. v33}, pages 495--514.

\bibitem{dyn2003subdivision}
N.~Dyn and D.~Levin.
\newblock {Subdivision schemes in geometric modelling}.
\newblock {\em Acta Numerica}, 11:73--144, 2003.

\bibitem{dyn19874}
N.~Dyn, D.~Levin, and J.~Gregory.
\newblock {A 4-point interpolatory subdivision scheme for curve design}.
\newblock {\em Computer Aided Geometric Design}, 4(4):257--268, 1987.

\bibitem{dyn2006approximations}
N.~Dyn and A.~Mokhov.
\newblock {Approximations of set-valued functions based on the metric average}.
\newblock {\em Rendiconti di Matematica}, 26:249--266, 2006.

\bibitem{erdos1945some}
P.~Erd{\"o}s.
\newblock {Some remarks on the measurability of certain sets}.
\newblock {\em Bull. Amer. Math. Soc}, 51(10):728--731, 1945.

\bibitem{frechet1948elements}
M.~Fr{\'e}chet.
\newblock Les {\'e}l{\'e}ments al{\'e}atoires de nature quelconque dans un
  espace distanci{\'e}.
\newblock {\em Ann. Inst. Henri Poincar{\'e},(10)}, pages 215--310, 1948.

\bibitem{galton2000qualitative}
A.~Galton.
\newblock {\em Qualitative spatial change}.
\newblock Oxford University Press Oxford,, UK, 2000.

\bibitem{halmos1974measure}
P.~Halmos.
\newblock {\em {Measure theory}}.
\newblock Springer, 1974.

\bibitem{herman1992shape}
G.~Herman, J.~Zheng, and C.~Bucholtz.
\newblock {Shape-based interpolation}.
\newblock {\em IEEE Computer Graphics and Applications}, 12(3):69--79, 1992.

\bibitem{jones2001lebesgue}
F.~Jones.
\newblock {\em Lebesgue integration on Euclidean space}.
\newblock Jones \& Bartlett Learning, 2001.

\bibitem{kels2011reconstruction}
S.~Kels and N.~Dyn.
\newblock Reconstruction of 3d objects from 2d cross-sections with the 4-point
  subdivision scheme adapted to sets.
\newblock {\em Computers \& Graphics}, 2011.

\bibitem{lane1980theoretical}
J.~Lane and R.~Riesenfeld.
\newblock A theoretical development for the computer generation and display of
  piecewise polynomial surfaces.
\newblock {\em Pattern Analysis and Machine Intelligence, IEEE Transactions
  on}, (1):35--46, 1980.

\bibitem{levin1986multidimensional}
D.~Levin.
\newblock {Multidimensional reconstruction by set-valued approximations}.
\newblock {\em IMA Journal of Numerical Analysis}, 6(2):173, 1986.

\bibitem{mattila1999geometry}
P.~Mattila.
\newblock {\em Geometry of sets and measures in Euclidean spaces: Fractals and
  rectifiability}, volume~44.
\newblock Cambridge Univ Pr, 1999.

\bibitem{muresanset}
M.~Muresan.
\newblock {Set-valued approximation of multifunctions}.
\newblock {\em Studia Univ. Babes-Bolyai, Mathematica 55}, pages 107--148,
  2010.

\bibitem{raya1990shape}
S.~Raya and J.~Udupa.
\newblock {Shape-based interpolation of multidimensional objects}.
\newblock {\em IEEE transactions on medical imaging}, 9(1):32--42, 1990.

\bibitem{sabin2010analysis}
M.~Sabin.
\newblock {\em Analysis and Design of Univariate Subdivision Schemes},
  volume~6.
\newblock Springer Verlag, 2010.

\bibitem{vitale1979approximation}
R.~Vitale.
\newblock Approximation of convex set-valued functions.
\newblock {\em Journal of Approximation Theory}, 26(4):301--316, 1979.

\end{thebibliography}

\end{document}